\documentclass[a4paper,12pt,reqno]{amsart}

\usepackage{amsmath}
\usepackage{amsmath}
\usepackage{mathrsfs}
\usepackage{amsfonts}
\usepackage{graphicx}
\usepackage{color}
\usepackage{amsfonts}
\usepackage{amssymb}
\usepackage{fancyhdr}
\usepackage{psfrag}
\usepackage[latin1]{inputenc}
\usepackage{mathrsfs}

\usepackage[hidelinks]{hyperref}

\renewcommand\eqref[1]{(\ref{#1})} 
\graphicspath{ {images/} }
\usepackage[latin1]{inputenc}
\usepackage{psfrag}
\usepackage{comment}
\usepackage{enumerate}

\newtheorem{theorem}{Theorem}[section]

\newtheorem{lemma}[theorem]{Lemma}
\newtheorem{definition}[theorem]{Definition}

\newtheorem{remark}[theorem]{Remark}
\newtheorem{corollary}[theorem]{Corollary}

\newtheorem{thmprime}{Theorem}

\numberwithin{equation}{section}

\def\be#1 {\begin{equation} \label{#1}}
	\newcommand{\ee}{\end{equation}}

\def\sqw{\hbox{\rlap{\leavevmode\raise.3ex\hbox{$\sqcap$}}$%
		\sqcup$}}
\def\findem{\ifmmode\sqw\else{\ifhmode\unskip\fi\nobreak\hfil
		\penalty50\hskip1em\null\nobreak\hfil\sqw
		\parfillskip=0pt\finalhyphendemerits=0\endgraf}\fi}

\newcommand{\R}{{\mathbb {R}}}

\newcommand{\N}{{\mathbb N}}

\newcommand{\C}{{\mathbb C}}

\newcommand{\CL}{\mathcal L}

\newcommand{\supp}{\operatorname{supp}}

\usepackage[text={6in,8.6in},centering]{geometry}
\author{Riju Basak and K. Jotsaroop}
\address{
}
\email{}

\address [Riju Basak]{
	Department of Mathematics\\
	Indian Institute of Science Education and Research\\
	Mohali, India \\
	\newline
	and
	\newline
	Department of Mathematics, National Taiwan Normal University, Taipei, Taiwan
}
\email{rijubasak52@gmail.com}

\address[Jotsaroop Kaur]{
	Department of Mathematics\\
	Indian Institute of Science Education and Research\\
	Mohali, India}
\email{jotsaroop@iisermohali.ac.in}

\keywords{Hardy spaces, Maximal function, wave operator, twisted Laplacian}
\subjclass[2010]{Primary 42A85, 42B15, Secondary 42B25}

\begin{document}
	
	\title[Hardy spaces and Wave operator]{On Hardy Spaces associated with the Twisted Laplacian and sharp estimates for the corresponding Wave Operator}
	
	\begin{abstract} We prove various equivalent characterisations of the Hardy space $H^p_{\CL}(\C^n)$ for $0<p<1$ associated with the twisted Laplacian $\mathcal{L}$ which
		generalises the result of  \cite{Mauceri-Picardello-Ricci-Hardy-space-1981} for the case $p=1$. Using the atomic characterisation of $H^p_{\CL}(\C^n)$ corresponding to the twisted convolution, we prove sharp boundedness result for the wave operator $\CL^{-\delta/2}e^{\pm it\sqrt{\CL}}$ for a fixed $t>0$ on $H^p_{\CL}(\C^n)$. More precisely we prove that it is a bounded operator from $H^p_{\CL}(\C^n)$ to $L^p(\C^n)$ for $ 0<p\leq 1$ and $\delta\geq (2n-1)\left(1/p-1/2\right)$. 
	\end{abstract}

	\setcounter{tocdepth}{2}

	\maketitle
	
	\section{Introduction and Main results}
	
	The theory of classical Hardy spaces on $\R^n$ corresponding to Euclidean convolution structure is a well studied object.
	In the seminal work of C. Fefferman and E.M. Stein (\cite{FS}), the real variable theory of classical Hardy spaces was developed involving it's various equivalent characterisations. An atomic decomposition of the classical Hardy space was obtained by 
	R. Coifman \cite{Coif1} when $n=1$ and R. Latter \cite{Latter} for $n>1$.

	Let $H^p(\R^n),0<p<\infty$ denote the classical Hardy space. Define $$\mathcal{C}=\left\{\varphi\in \mathcal{S}(\R^n): \sup_{|\alpha|\leq N}\int_{\R^n}(1+|x|^2)^M|\partial^{\alpha}\varphi(x)|dx\leq 1\right\}.$$ Here $N,M$ depend on $p,n$ and $\alpha=(\alpha_1,\ldots,\alpha_n); \alpha_i\in \N\cup \{0\}$ for $i=1$ to $n$. Define $\varphi_t(x)=t^{-n}\varphi(t^{-1}x),t>0$.
	The following characterisation is known:
	\begin{align}
		\nonumber f\in H^p(\R^n)	\iff& f\in \mathcal{S}'(\R^n),\sup_{t>0}|f*\varphi_t|\in L^p(\R^n),\varphi\in \mathcal{S}(\R^n),\int_{\R^n}\varphi \neq 0\\
		\nonumber \iff& f\in \mathcal{S}'(\R^n),\sup_
		{\varphi\in \mathcal{C}}\sup_{t>0}|f*\varphi_t|\in L^p(\R^n)	\\
		\label{eqn:EC}\iff& f\in \mathcal{S}'(\R^n), \text{there exists $p$-atoms}~ \{a_j\}_{j\geq 1} ~ \text{such that}\\ \nonumber \quad & f=\sum_{j\geq 1}\lambda_j a_j~ \text{in}~ \mathcal{S}'(\R^n)~ \text{and}~ \sum_{j\geq 1} |\lambda_j|^p<\infty.
	\end{align}
	We say $a$ is a $p$-atom if 	$\supp a\subset Q(x_0,r),r>0$ where $Q(x_0,r)$ is a cube centered at $x_0$ and length $r$, $\|a\|_{\infty}\leq r^{-n/p}$ and $\int_{\R^n}x^{\alpha}a(x)dx=0$ for all $|\alpha|\leq \lfloor n(1/p-1) \rfloor$. Here $\lfloor\cdot\rfloor$ is the greatest integer function.
	Moreover there exist constants independent of the choice of $f$ such that
	\begin{align*}
		\|\sup_{t>0}|f*\varphi_t| \|_p\simeq_{p,n}\sup_
		{\varphi\in \mathcal{C}}\sup_{t>0}\|f*\varphi_t\|_p
		\simeq_{p,n} \inf \left(\sum_{j\geq 1} |\lambda_j|^p\right)^{\frac{1}{p}}
	\end{align*}
	for $\varphi\in\mathcal{C}$.
	The infimum above is taken over all possible representations of $f$ in terms of $p$-atoms in the sense of \eqref{eqn:EC}.
	When $p>1$ the corresponding Hardy space $H^p(\R^n)$ coincides with the usual Lebesgue space $L^p(\R^n)$. 
	



	All the above characterisations are very useful to study the boundedness of Fourier multiplier operators when dealing with the case $0<p\leq 1$. For example Miyachi \cite{Miyachi-wave} proved boundedness of the wave operator on $H^p(\R^n)$ using the atomic decomposition. These characterisations are also useful in studying singular integral operators on $H^p(\R^n)$ when $0<p\leq 1$.



	
	In this article one of our main goals is to study Hardy space theory for $p<1$ corresponding to the twisted Laplacian $\mathcal{L}$ on $\C^n\equiv\R^{2n}$ defined as \begin{align*}
		\mathcal{L}=-\sum_{j=1}^n\left(\frac{\partial}{\partial x_j}+\frac{1}{2}iy_j\right)^2+\left(\frac{\partial}{\partial y_j}-\frac{1}{2}ix_j\right)^2.	
	\end{align*}
	
	The operator $\mathcal{L}$ is a second order elliptic differential operator on $\R^{2n}$ and essentially self adjoint operator on $L^2(\R^{2n})$. The operator $\mathcal{L}$ is connected to the sub-Laplacian on the Heisenberg group which is the sum of squares of the generators of left invariant vector fields on the Heisenberg group (see Section~\ref{Preliminaries}). 

	The twisted Laplacian $\mathcal{L}$ can also be viewed as a Schr\"{o}dinger operator with constant magnetic field in quantum mechanics as well. See \cite{RT-cont-magnetic},\cite{Reed-Simon}. 
	The spectral decomposition of $\mathcal{L}$ is explicitly known. It has discrete spectrum on the positive real line and is bounded away from the origin {(see Section \ref{Preliminaries})}. 
	Define $Z_j= \frac{\partial}{\partial z_{j}}-\frac{1}{2}\bar{z}_j$, $\bar{Z}_j= \frac{\partial}{\partial \bar{z}_{j}}+\frac{1}{2} z_j$, $j=1,2,\ldots, n$. Here $\frac{\partial}{\partial z_{j}}=\frac{\partial}{\partial x_{j}}-i\frac{\partial}{\partial y_{j}}$ and $\frac{\partial}{\partial \bar{z}_{j}}=\frac{\partial}{\partial x_{j}}+i\frac{\partial}{\partial y_{j}}$. Using the above expression for $\{Z_j,\bar{Z}_j\}_{j=1}^n$, we can write 
	\begin{align*}
		\mathcal{L}= - \frac{1}{2}\sum_{j=1}^{n} \left(Z_j \bar{Z}_j + \bar{Z}_j Z_j\right).
	\end{align*}
	Using the expression for $Z_j$ and $\bar{Z}_j$ above, it is easy to check that \begin{align*}[Z_j,\bar{Z}_j]=2 \text{Id}, ~[Z_j,Z_k]=0,
		[Z_j,\bar{Z}_k]=0, j\neq k\end{align*} for $j,k=1,2,\ldots,n$.
Due to the connection with the Heisenberg group, there is an induced twisted convolution on $\R^{2n}$ with which the operator $\mathcal{L}$ commutes.  

Define $$\tau_{w}f(z)=f(z-w)e^{\frac{i}{2}Im(z.\bar{w})},$$
twisted translation of $f$ by $w$. See  Section \ref{Preliminaries} for the connection of twisted translation with the Heisenberg group.
The twisted convolution of two functions $f$ and $g$ is defined by 
\begin{align*}
	(f\times g)(z)=\int_{\mathbb{\R}^{2n}} f(w)g(z-w) e^{\frac{i}{2}Im(z\cdot \bar{w})} \, dw= \int_{\mathbb{C}^n} f(w)\tau_w g(z)  \, dw.
\end{align*}
In this article, we are interested to obtain a characterisation of the Hardy space associated with the intrinsic twisted convolution structure arising from $\CL$ for the range $0 < p < 1$. 
When $p=1$, such a characterisation of the Hardy space associated with $\mathcal{L}$ was investigated in \cite{Mauceri-Picardello-Ricci-Hardy-space-1981}. 
One of the motivation to obtain various characterisations of the corresponding Hardy space is to prove sharp boundedness result for the wave operator associated to the twisted Laplacian $\mathcal{L}$. 
Let $A$ be a non-negative essentially self adjoint operator on $L^2(\R^n)$. The heat semigroup $e^{-tA}, t>0$ associated to $A$ is defined via the functional calculus as follows 
\begin{align*}
	e^{-tA}= \int_{0}^{\infty} e^{-t\lambda} dE_{\lambda}
\end{align*}
where $dE_{\lambda}$ is the spectral measure associated to $A$. Let us assume that $e^{-tA}$ maps $\mathcal{S}(\R^n)$ to itself. Let us define the Hardy space corresponding to  $A$ via the heat semigroup $e^{-tA}, t>0$ as the following
$$H^p_{A}(\R^n):=\{f\in \mathcal{S}'(\R^n): \sup_{t>0}|e^{-tA}f|\in L^p(\R^n)\}$$ for $0<p<\infty$.	
Under the assumption that the heat kernel for the operator $e^{-tA}$ satisfies Gaussian estimates, Song and Yan \cite{SongYan} proved various equivalent characterisations of $H^p_{A}(\R^n)$ including an atomic characterisation corresponding to the operator $A$.	
In the non-Euclidean setting, G.B. Folland and E.M. Stein developed the real variable theory of Hardy spaces on homogeneous groups in \cite{Folland-Stein-book} including the atomic decomposition. The theory of  Hardy spaces corresponding to a non-negative self adjoint operator on a doubling metric space has been studied by various authors. A suitable atomic characterisation of such spaces is known if the heat kernel corresponding to the operator  satisfies Gaussian type decay estimates. See \cite{Dzuibanksy-Hardy-Schrodinger,BDT-Hardy-Bessel,HLMMY-Hardy-Memoirs,Duong-Li-Hardy,SongYan,HHLLYY-2019,Bui-Duong-Ly-Hardy-finitemeasure} and references therein.	

Using functional calculus for the operator $\mathcal{L}$ we define the heat semigroup $e^{-t\mathcal{L}},t>0$ corresponding to $\mathcal{L}$. 
The operator $f\rightarrow e^{-t\mathcal{\CL}}f$ solves the heat equation corresponding to $\CL$ for $t>0$ with the initial data at $t=0$ being $f$. 
We define the Hardy space $$H^p_{\mathcal{L}}(\C^n)= \biggl\{ f\in \mathcal{S}'(\C^n) : M_{\CL}^{\text{heat}}f(z)= \sup_{0<t<\infty}\left| e^{-t\mathcal{L}}f(z)\right| \in L^p(\mathbb{C}^n)\biggr\}.$$ $\mathcal{H}^{p}_{\mathcal{L}}
(\mathbb{C}^n)$ is equipped with the norm $\|f\|_{H^p_{\mathcal{L}}(\C^n)}=\|M_{\CL}^{\text{heat}}f\|_p$.

Let us now define atoms associated to $\CL$. Let $\omega(z,w):=e^{\frac{i}{2}Im(z\cdot \bar{w})}$ on $\C^n\times \C^n$.	 Note that  $z\cdot\overline{w}=\sum_{j=1}^n z_j\overline{w}_j$. \begin{definition}\label{Def:atom}
	Let $0<p\leq 1$ and $0<\sigma < \infty$. We call a measurable function $f$ a $(p, \sigma)$-atom if there exists a cube $Q=Q(z_0,r)$ centered at $z_0$ and length $r$ such that 
	\begin{enumerate}
		\item $\supp f \subset Q$,
		\item $\|f\|_{\infty}\leq r^{-2n/p}$,
		\item $\int f(z) z^{\alpha} \bar{z}^{\beta} \omega(z_0,z) \, dz =0$ for all $|\alpha|+|\beta| \leq \mathcal{N}_{0},$
		whenever $r<\sigma$.  Here $\mathcal{N}_{0}=\lfloor 2n(1/p-1) \rfloor $ is a fixed number.
	\end{enumerate}	
\end{definition}

Let us consider a collection of functions
$$\mathcal{S}_{N}= \{ \varphi \in C^{\infty}(\R^{2n}): \supp \varphi \subset Q(0,1) \, \text{and} \, |\partial^{\alpha}\varphi| \leq 1 \, \text{for} \, \text{all} \, |\alpha|\leq N \}.$$
Here $Q(0,1)=[-1/2,1/2]^{2n}$ is the cube of side length $1$ centered at $0$.	 Let $\varphi_{t}(z):= t^{-2n}\varphi(z/t)$.	
The choice of $N$ is only dependent on $n$ and $p$. We will specify it in Section \ref{section:Hardy space}.

We prove the following theorem:
\begin{thmprime}\label{MainThmHardy}Let $f\in\mathcal{S}'(\R^n)$. For any $0<p<1$, the following are equivalent
	\begin{enumerate}
		\item [i)] $\sup_{\varphi\in \mathcal{S}_{N}}\sup_{0<t<1}|f\times \varphi_t|\in L^p(\C^n)$.
		\item [ii)]	$\sup_{t>0}|e^{-t\mathcal{L}}f|\in L^p(\C^n)$. 
		\item [iii)] Define $Tf(z,u)=f(z)e^{i u}$. $Tf\in H^p({\mathbb{H}^n_{\text{red}}})$.
		\item [iv)] $f=\sum_{j\geq 1}\lambda_j a_j$ in $\mathcal{S}'(\R^{2n})$, where $a_j$'s are $(p,1)$-atoms and $\sum_j|\lambda_j|^p<\infty$.	\end{enumerate}

	
	Moreover, $\|\sup_{\varphi\in \mathcal{S}_N}\sup_{0<t<1}|f\times \varphi_t|\|_p^p\simeq_{p}\|\sup_{t>0}|f\times p_t|\|_p^p\simeq_p \inf\{\sum_{j\geq 1}|\lambda_j|^p:f=\sum_{j\geq 1}\lambda_j a_j,~~\text{where}~ a_j ~~\text{are $(p,1)$-atoms} \}$.
	
\end{thmprime}
See Remark \ref{rem:reduced} for definition of $H^p({\mathbb{H}^n_{\text{red}}})$.


When $\frac{2n}{2n+1}< p\leq 1$, a similar characterisation as above in Theorem \ref{MainThmHardy} has been obtained in \cite{Huang-Wang-2014}. In this range of $\frac{2n}{2n+1}< p\leq 1$, $\mathcal{N}_0=0$, therefore the proof of such a characterisation is similar to $p=1$ case. For the atomic characterisation of the Hardy space associated to $\mathcal{L}$ with higher cancellation condition, the methods in \cite{Mauceri-Picardello-Ricci-Hardy-space-1981} do not directly apply. When $p=1$, in \cite{Huang-2009} a characterisation of the Hardy space for twisted Laplacian in terms of Littlewood-Paley g-function has also been obtained.

We would also like to mention that even though the heat kernel corresponding to $\mathcal{L}$ satisfies Gaussian estimates, we cannot directly use the  the atomic space characterisation by Song-Yan in \cite{SongYan,Bui-Duong-Ly-Hardy-finitemeasure} to prove atomic space characterisation for twisted Laplacian $\mathcal{L}$ which is natural to the intrinsic structure of twisted convolution and also to obtain sharp boundedness results for the wave operator corresponding to $\mathcal{L}$.  
\begin{theorem}\label{Thm:HARDY2} Let $H^p_{\mathcal{L}}(\C^n)$ be as defined above corresponding to the twisted Laplacian $\mathcal{L}$. Then $f\in H^p_{\mathcal{L}}(\C^n)$ if and only if $f$ satisfies any of the equivalent conditions in Theorem \ref{MainThmHardy}. Moreover, we also have the norm equivalence.
	
\end{theorem}

The wave operators associated with $\mathcal{L}$ are oscillatory spectral multiplier operators given by $e^{\pm it\sqrt{\mathcal{L}}}$, $t>0$. It  also has a connection with the solution of Cauchy problem for the wave equation associated to $\mathcal{L}$. The Cauchy problem for the wave equation corresponding to $\mathcal{L}$ is given by
\begin{align}\label{eq:Cauchy-wave}
	\nonumber	\frac{\partial^2u}{\partial t^2}(x,t)+&\mathcal{L} u(x,t)=0, \, \, x\in \mathbb{R}^{2n}, t> 0,\\
	u(x,0)&=f(x), \, x\in \mathbb{R}^{2n},\\
	\nonumber \frac{\partial u}{\partial t}(x,0)&=g(x), \, x\in \mathbb{R}^{2n}.
\end{align}
Using the functional calculus for $\mathcal{L}$, the solution of \eqref{eq:Cauchy-wave} can be written  as follows 
\begin{align*}
	u(x,t)=\cos(t\sqrt{\mathcal{L}})f(x)+\mathcal{L}^{-1/2} \sin(t\sqrt{\mathcal{L}})g(x)
\end{align*}
for $x\in \mathbb{R}^{2n}$ and $t>0$ for $f,g\in L^2(\R^{2n})$. To study the $L^p$ mapping properties of $(f,g)\rightarrow u(\cdot, t)$ for fixed $t>0$, it is enough to estimate the following oscillatory spectral multiplier operators 
\begin{align*}
	\mathcal{L}^{-\delta/2} e^{\pm it\sqrt{\mathcal{L}}}
\end{align*}
for $\delta\geq 0$ and $t>0$.
We prove the following sharp result concerning the boundedness of $\mathcal{L}^{-\delta/2} e^{\pm it\sqrt{\mathcal{L}}}$. 
\begin{thmprime}\label{thm:Main-result}
	The operator $\mathcal{L}^{-\delta/2} e^{\pm it\sqrt{\mathcal{L}}}$ is bounded from $H^p_\mathcal{L}(\mathbb{C}^n)$ to $L^p(\C^n)$ for $0<p\leq 1$ with $\delta \geq (2n-1)\left(1/p-1/2\right).$
\end{thmprime}
On interpolating between $H^1_\mathcal{L}(\C^n)$ and $L^2(\C^n)$ and using duality we obtain the following corollary.
\begin{corollary}\label{Cor:Main thm}
	The operator $\mathcal{L}^{-\delta/2} e^{\pm it\sqrt{\mathcal{L}}}$ is bounded on $L^p(\C^{n})$ for $1<p<\infty$ and $\delta\geq (2n-1)|1/p-1/2|$.	
\end{corollary}

The results in Theorem \ref{thm:Main-result} and Corollary \ref{Cor:Main thm} are sharp with respect to the parameter $\delta$. The sharpness of $\delta$ in Corollary \ref{Cor:Main thm} can be proved using transplantation result due to Kenig-Stanton-Tomas in \cite{Kenig-Stanton-Tomas} and the sharpness of classical Euclidean results proved by Miyachi \cite{Miyachi-wave} and Peral \cite{Peral-wave}. The sharpness of $\delta$ in Theorem \ref{thm:Main-result} can be argued using interpolation and the sharpness of the Corollary \ref{Cor:Main thm}.
\begin{remark}
	\begin{enumerate}[(i)]
		\item In the article \cite{Naru-Thangevelu-2001}, E.K. Narayanan and S. Thangavelu studied the boundedness of the wave operators associated with the twisted Laplacian $\mathcal{L}$. They proved $L^p$ boundedness result for $1<p<\infty$ and $\delta>(2n-1)|1/p-1/2|$ (see \cite[Corollary 1]{Naru-Thangevelu-2001}). Therefore the Corollary \ref{Cor:Main thm} is the improvement of the result of \cite[Corollary 1]{Naru-Thangevelu-2001}. 
		\item In the article \cite{Bui}, the authors prove a general boundedness of oscillatory spectral multipliers of the form $m(\lambda)=(1+\lambda)^{-\delta}e^{i\lambda^{\alpha}},\alpha,\delta>0$ corresponding to a non negative self adjoint operator on $L^2(\R^n)$ whose heat kernel satisfies Gaussian type estimates. They prove that the oscillatory spectral multiplier operator is bounded on a suitable Hardy space if $\delta\geq  n\alpha|1/p-1/2|$. See Theorem 1.2 in \cite{Bui}.
	\end{enumerate}
\end{remark}


In the Euclidean setting, the sharp estimate for the wave operator on the Euclidean Hardy spaces $H^p(\R^n)$ was studied by A. Miyachi \cite{Miyachi-wave}. It was also proved independently by J.C. Peral \cite{Peral-wave} for $1\leq p <\infty$. The sharp estimate for the wave operators for the range $1\leq p <\infty$ can also be concluded from the deep result of Seeger-Sogge-Stein \cite{Seeger-Sogge-Stein}. In \cite{Seeger-Sogge-Stein}, the results are true for more general Fourier integral operators on compact manifolds.

The boundedness of wave operators has been studied in various other settings as well. D. M\"uller and E.M. Stein \cite{Muller-Stein-wave}, and later D. M\"uller and A. Seeger \cite{Segeer-Muller-wave-APDE-2015}, studied the boundedness of wave operators on the Heisenberg group $\mathbb{H}^d$($m=2d+1$ dimensional) and Heisenberg type groups. For more general two-step Carnot groups, A. Martini and D. M\"uller \cite{Martini-Muller-wave} have recently investigated the boundedness of wave operators on $L^p,p\geq 1$. In the context of compact Lie groups, the boundedness of the wave operator was examined by J. Chen, D. Fan, and L. Sun \cite{Chen-Fan-Sun-compact-wave}. Beyond group settings, the boundedness of wave operators associated with various self-adjoint operators on metric measure spaces have also been analysed. Important examples include studies on the Grushin operator \cite{Kaur-Thangavelu-Grushin-wave}, the Hermite operator \cite{Bui-Duong-Hermite-IMRN}, and general self-adjoint operators whose heat semigroup kernel satisfies Gaussian-type decay \cite{Bui}.

While many studies have been done for wave operators but the complete understanding of sharp fixed-time estimates is still unknown in many settings. To the best of our knowledge, aside from classical Euclidean settings and compact Lie groups,  the sharp fixed time estimates for wave operators on Hardy spaces in the range $0<p<1$ are not known. In this article, we will study and prove the sharp fixed time estimates of the wave operators on Hardy spaces associated with the twisted Laplacian $\mathcal{L}$ for the entire range $0<p<\infty$.

The summary of our paper is as follows. In Section~\ref{Preliminaries}, we provide background on the Heisenberg group and establish Taylor's formula for the twisted translation. Section~\ref{section:Hardy space} is devoted to proving the equivalence of various maximal functions in $L^p$ norms and to establishing the atomic characterisation of the Hardy space associated with the twisted Laplacian, which leads to the proof of Theorem~\ref{MainThmHardy}. Theorem~\ref{thm:Main-result} is proved in Section~\ref{wave}. Finally, in Section~\ref{Sec:Appendix}, we present some technical results that are used in the proof of the main theorems.

\medskip \noindent \textbf{Notations:} For two non-negative numbers $M_1$ and $M_2$, by the expression $M_1\lesssim M_2$, we mean there exists a constant $C>0$ such that $M_1\leq C M_2$. Whenever the implicit constant dependent on $\epsilon$, we write $M_1\lesssim_{\epsilon}M_2$. We write $M_1\simeq M_2$, when $M_1\lesssim M_2$ and $M_2\lesssim M_1$. We define the multi-index $\alpha=(\alpha_1,\ldots, \alpha_n)$ where $\alpha_i\in \mathbb{N}\cup\{0\}$ and $\alpha!:= \alpha_1!\ldots \alpha_n!$.

	\section{Preliminaries}\label{Preliminaries} 
	
	In this section we will state some important properties of the twisted Laplacian $\mathcal{L}$ which will be used later in the proof of our main results. Let us first introduce the Heisenberg group.
	We define $\mathbb{H}^n:=\mathbb{C}^{n}\times \mathbb{R}$, the $(2n+1)$ dimensional Heisenberg group with the group law defined by 
	\begin{align*}
		(z,t)\circ (w,s)= \left(z+w, t+s+\frac{1}{2}Im(z\cdot \bar{w})\right) \, \text{for}~ \,  (z,t),(w,s)\in \mathbb{H}^n.
	\end{align*}

	{The subspace $\{(0,t):t\in\R\}$ is the center of $\mathbb{H}^n$}.
	Let $$\mathcal{X}_j:=\frac{\partial}{\partial x_j}+\frac{1}{2}y_j\frac{\partial}{\partial t}$$
	and $$\mathcal{Y}_j:=\frac{\partial}{\partial y_j}-\frac{1}{2}x_j\frac{\partial}{\partial t}.$$ It is known that $\{\mathcal{X}_j,\mathcal{Y}_j\}_{j=1}^n$ and $\mathcal{T}=-\frac{\partial}{\partial t}$ is a basis of Lie algebra of left invariant vector fields on $\mathbb{H}^n$. Moreover $\mathcal{X}_j$ and $\mathcal{Y}_j$ satisfy the commutation relations $[\mathcal{X}_j,\mathcal{Y}_j]=\mathcal{T}$ for all $j=1,..,n$. This implies that $\{\mathcal{X}_j,\mathcal{Y}_j\}_{j=1}^n$ generate the Lie algebra of left invariant vector fields on $\mathbb{H}^n$.
	
	We define $\CL_{\mathbb{H}^n}=- \sum_{j=1}^n \left(\mathcal{X}_j^2+\mathcal{Y}_j^2\right)$, the sub-Laplacian on the Heisenberg group. It is known that $\CL_{\mathbb{H}^n}$ is a hypoelliptic operator. For more details on the Heisenberg sub-Laplacian please refer to \cite{Stein-book-1993} and \cite{Thangavelu-book-uncertainty}. Let $X_j(\lambda),Y_j(\lambda)$ be vector fields on $\R^{2n}$ satisfying $$\mathcal{X}_j(f(z)e^{i\lambda t})=X_j(\lambda)f(z)e^{i\lambda t}$$  and
	$$\mathcal{Y}_j(f(z)e^{i\lambda t})=Y_j(\lambda)f(z)e^{i\lambda t}$$ for $\lambda\neq 0$. More expicitely, for $\lambda\neq 0$ and $j=1,2,\ldots,n$,   $X_j(\lambda),Y_j(\lambda)$ are given by 
	$$ X_j(\lambda)=\frac{\partial}{\partial x_j}+\frac{i\lambda}{2}y_j, \, Y_j(\lambda)=\frac{\partial}{\partial y_j}-\frac{i\lambda}{2}x_j.$$
	
	It is easy to verify that the twisted Laplacian can be written as $$\CL=-\sum_{j=1}^n \left(X_j(1)^2+ Y_j(1)^2\right).$$ For our convenience, from now on we will write $X_j$ for $X_j(1)$, $Y_j$ for $Y_j(1)$, $\tilde{X}_j$ for $X_j(-1)$ and $\tilde{Y}_j$ for $Y_j(-1)$, for all $j=1,2,\ldots,n$.

	Recall that $\tau_{w}f(z)=f(z-w)e^{\frac{i}{2}Im(z.\bar{w})}$ is the twisted translation of $f$ by $w$. In fact, $\tau_{w}$ defined above is indeed the left translation by $(-w,0)$ on the Heisenberg group restricted on the functions of type $f(z)e^{it}$ and evaluated at $t=0$. 
	
	Using the connection of $\CL$ with the  sub-Laplacian $\CL_{\mathbb{H}^n}$ on $\mathbb{H}^n$, it is easy to check that $\CL(\tau_{w}f)(z)=\tau_{w}\CL f(z)$. Consequently, we also have $\CL(f\times g)(z)=  f\times  \CL g(z)$ for all $f,g$ in Schwartz class on $\R^{2n}$. 
	
	The operator $\mathcal{L}$ is a positive self-adjoint second-order elliptic differential operator. The spectrum of $\mathcal{L}$ is  the set $\Gamma:=\{(2k+n):k\in \N\cup\{0\}\}$.
	It is known that $\CL \varphi_k= (2k+n) \varphi _k$ for $k\geq 0$, where $\varphi_k$ are Laguerre functions of order $n-1$. 
	
	The spectral resolution of identity of  $\CL$ is given by $f=(2\pi)^{-n}\sum_{k\geq 0} f\times \varphi_k  $ for $f\in L^2(\C^n)$. We also know that \begin{equation}\label{PT}\int_{\C^n}|f(z)|^2 \,dz=(2\pi)^{-2n}\sum_{k\geq 0}\int_{\C^n}|f\times \varphi_k(z)|^2 \,dz\end{equation} where with abuse of notation $dz$ is the usual Lebesgue measure on $\R^{2n}$.
	
	Let $m$ be a bounded measurable function on $[n,\infty)$. We define the corresponding multiplier operator $m(\CL)f =(2\pi)^{-n}\sum_{k\geq 0}m(2k+n) f\times\varphi_k  $ for $f\in L^2(\C^n)$. When $m(\lambda)=\lambda^{-\delta/2}e^{\pm it\sqrt{\lambda}},\delta\geq 0$, we get the wave operator associated to $\mathcal{L}$.
	\subsection{Taylor expansion for twisted translation}	
	In this subsection, we will first state a Taylor expansion formula for the case of Heisenberg group. Then, we will derive the corresponding formula for twisted translations as a special case from the Taylor's expansion formula on the Heisenberg group.
	
	
	For any smooth function $F$ in $\mathbb{H}^n$, the Taylor series expansion w.r.t. left invariant vector fields of $F$ about the point $(z,t)$ is given by 
	\begin{align}\label{eq:Taylor-formula-Heisenberg}
		F((z,t)\circ(w,s))  = & \sum_{j\leq N} \frac{1}{j!}\left(Log(w,s)\right)^j F(z,t)  \\
		\nonumber +  \frac{1}{N!}&\int_{\theta=0}^{1}(1-\theta)^{N} \left(Log(w,s)\right)^{N+1}F((z,t)\circ (Exp (\theta Log( w, s)))  \, d\theta.
	\end{align}
	Note that $z=x+iy,x,y\in\R^n$.
	For the reference of the above Taylor's formula, see Lemma $20.3.8$ and Corollary $20.3.9$ in \cite{AB}. 
	In \cite{AB} the Taylor expansion formula is given in terms of the $Log$ and $Exp$ map. In the case of the Heisenberg group with the group operation as defined above $Log(z,t)=\sum_{i=1}^n x_i\mathcal{X}_i + y_i\mathcal{Y}_i+t \mathcal{T}$ and $Exp(Log(z,t))=(z,t)$.
	
	Let us define $X_{j+n}=Y_j$ and similarly for $\tilde{X}_{j+n}=\tilde{Y}_j$ for $j=1,2,\ldots,n$.
	Recall that $\mathcal{X}_j\left(f(z)e^{-it}\right)=\tilde{X}_jf(z)e^{-it}$ for all $j=1,2,..n$. Similar identity also follows for ${\mathcal{Y}_j}$ and {$\tilde{Y}_j$} above respectively. In fact for any power of $\mathcal{X}_j$ and $\mathcal{Y}_j$ on $F(z,t)=f(z)e^{it},$ we get the corresponding power of $X_j$ on $f$ times $e^{it}$. Similarly for $F(z,t)=f(z)e^{-it}$ with $X_j$'s replaced by $\tilde{X}_j$.
	
	
	
	Now we take the function $F(z,t)= f(z)e^{-it}$, put $s=0,w\in\C^n$ in the above expansion \eqref{eq:Taylor-formula-Heisenberg} and evaluate at $(z,0)$ 
	we get 
	\begin{align}\label{eq:Taylor-formula-twisted} 
		f(z-w)e^{\frac{i}{2}Im( z\cdot\bar{w})} & = \sum_{k=1}^N\sum_{i_1,i_2,\ldots i_k=1}^{2n} \frac{(-1)^{k}}{k!}T_{i_1,i_2,\ldots i_k}(w)f(z) \\
		\nonumber + \sum_{1\leq i_1,i_2,\ldots i_{N+1}\leq 2n } \frac{(-1)^{N+1}}{N!}& \int_{\theta=0}^{1}(1-\theta)^{N}  T_{i_1,i_2,\ldots i_{N+1}}(w){f}(z-\theta w) 
		e^{\frac{i}{2}\theta Im(z\cdot \bar{w})} d\theta,
	\end{align} 
	where $$T_{i_1,i_2,\ldots i_k}(w)f(z)=u_{i_1}u_{i_2}\ldots u_{i_k}\tilde{X}_{i_1}\tilde{X}_{i_2}\ldots \tilde{X}_{i_{k}} {f}(z).$$  Let  us identify $w=v+iv'$ where $v,v'\in \R^n$ as an element in $\R^{2n}$  by the map $w\rightarrow (v_1,\ldots v_n,v_{1}',\ldots v'_{n})$. Here  $u_j=v_j,1\leq j\leq n$ and $u_{j+n}=v'_{j},1\leq j\leq n$ in terms of standard basis in $\R^{2n}$. 
	Let us define $U_{i_1,i_2,\ldots i_k}$ in a similar way as $T_{i_1,i_2,\ldots i_k}$ with $\tilde{X}$  replaced by $X$.
	The expansion \eqref{eq:Taylor-formula-twisted} is the Taylor expansion of $f$ corresponding to the twisted translation. 
	
	We also need the following version of Taylor expansion for the twisted translation.
	Consider $F(z,t)=f(z)e^{it}$. The Taylor expansion of $F$ at $(z,t)=(-a,0)$ and $(w,s)=(b,0)$ is
	\begin{align}\label{eq:Taylor-formula-twisted2} 
		f(b-a)e^{\frac{i}{2}Im( b\cdot\bar{a})} & = \sum_{k=1}^N\sum_{i_1,i_2,\ldots i_k=1}^{2n} \frac{1}{k!}U_{i_1,i_2,\ldots i_k}(b)f(-a) \\
		\nonumber + \sum_{1\leq i_1,i_2,\ldots i_{N+1}\leq 2n } \frac{1}{N!}& \int_{\theta=0}^{1}(1-\theta)^{N}  U_{i_1,i_2,\ldots i_{N+1}}(b){f}(-a+\theta b) 
		e^{\frac{i}{2}\theta Im(b\cdot \bar{a})} d\theta.
	\end{align} 
	{		\section{Hardy spaces and its atomic decomposition}\label{section:Hardy space}
		For a given $0<\sigma \leq \infty$, let us define a grand maximal function 
		\begin{align*}
			\mathcal{M}_{\sigma}^{N}f(z)=\sup_{\varphi \in \mathcal{S}_{N}} \sup_{0<t<\sigma} \left| f\times \varphi_{t}  (z) \right|.
		\end{align*}
		
		The choice of $N$ (sufficiently large) is a technical condition which only depends on $p$ and $n$.  So from now on we will fix $N$ and we shall write the above grand maximal function as $\mathcal{M}_{\sigma}$.  With abuse of notation when $\sigma = 1$, we write $\mathcal{M}_{1}$ as $\mathcal{M}$.
		
		For $0<\sigma \leq \infty$ and $0<p\leq 1$, we define the twisted Hardy space as 
		\begin{align*}
			H^p_{\mathcal{L},*,\sigma}(\mathbb{C}^n) := \left\{ f : f\in \mathcal{S}'(\R^{2n}), \|f\|_{H^p_{\CL,\sigma}(\mathbb{C}^n)}= \|\mathcal{M}_{\sigma}f\|_{L^p} < \infty \right\}.
		\end{align*}
		For $\sigma=1$, we write the Hardy space $H^p_{\mathcal{L},*,1}(\mathbb{C}^n)$ as $H^p_{\mathcal{L},*}(\mathbb{C}^n)$. We call this twisted Hardy space as the maximal function in its definition involves twisted convolution with functions from $\mathcal{S}_N$ defined in the introduction.
		
		
		
		Let us define the atomic Hardy space $H_{\CL, at, \sigma}^p(\mathbb{C}^n)$ as follows
		\begin{align*}
			H_{\CL, at, \sigma}^p(\mathbb{C}^n)= \left\{f= \sum_{j}c_{j}a_{j}: {a_{j}}'s \, \text{are} \, (p, \sigma)\, \text{-atom} \,\text{and} \, \sum_{j}|c_{j}|^{p}<\infty  \right\}  
		\end{align*}
		
		and the ``norm" in this space is defined by 
		\begin{align*}
			\|f\|_{H_{\CL, at, \sigma}^p(\mathbb{C}^n)} = \inf \left\{ \left(\sum_{j} |c_{j}|^p \right)^{1/p}: f= \sum_{j}c_{j}a_{j} \in H_{\CL, at, \sigma}^p(\mathbb{C}^n) \right\}.
		\end{align*}
		For the definition of $(p,\sigma)$ atom, see definition \ref{Def:atom}.
		Similarly, when $\sigma=1$, we just write the above space as $H_{\CL, at}^p(\mathbb{C}^n)$.
		We denote $H_{\CL, at, 1}^p(\mathbb{C}^n)$ as  $H_{\CL, at}^p(\mathbb{C}^n)$.
		We shall prove the following characterisation theorems for the Hardy spaces defined above for $0<p\leq 1$.
		
		One of the first characterisation theorem is: 
		\begin{theorem}\label{thm:charac-hardy-spaces} Let $f\in\mathcal{S}'(\R^{2n})$. For $0<p<1$,  $\mathcal{M}f\in L^p(\C^{n})$ if and only if $f\in H_{\CL, at}^p(\mathbb{C}^n)$. Moreover there exist universal constants $C,C'$ dependent on $p,n$ such that 
			$$C\|\mathcal{M}f\|_p^p\leq \sum_{j}|\lambda_j|^p\leq C'\|\mathcal{M}f\|_p^p$$
			where $f=\sum_j\lambda_j a_j$ as in the definition of $H_{\CL, at}^p(\mathbb{C}^n).$
		\end{theorem}
		When $p=1$, the above theorem was proved by G. Mauceri, M. A. Picardello and F. Ricci. See Theorem A in \cite{Mauceri-Picardello-Ricci-Hardy-space-1981}. When $p<1$, the analysis presented in \cite{Mauceri-Picardello-Ricci-Hardy-space-1981} is insufficient to prove Theorem \ref{thm:charac-hardy-spaces}. More specifically, when $ p < \frac{2n}{2n+1} $, the atoms satisfy higher-order cancellation conditions with a twisted translation depending on the support of the atom. To address this issue, we revisit the classical approach of R. Coifman \cite{Coif1} (see also \cite[Chapter 3]{Stein-book-1993}). In his classical approach, R. Coifman employed suitable projection operators to manage the higher-order cancellations of the atoms. 
		To achieve the appropriate atomic decomposition in our situation, we need to modify the projection operators to incorporate these twisted translations. See Lemma \ref{lem:center-atomic-decom}. 
		
		Let us define the Euclidean grand maximal function using the standard convolution as follows 
		\begin{align*}
			\widetilde{\mathcal{M}}_{\sigma}f(z)= \sup_{\varphi \in \mathcal{S}_{N}} \sup_{0<t<\sigma} |f\ast\varphi_{t}(z)|,
		\end{align*}
		where $f\ast g$ is the usual convolution on $\R^{2n}$ for a fixed sufficiently large $N$, and $0<\sigma \leq \infty$.
		
		Before proving Theorem \ref{thm:charac-hardy-spaces} we need to recall the definition of local Hardy space introduced by Goldberg in \cite{Goldberg-1979}.
		Let $h^p_{\sigma}(\mathbb{C}^n):=\{f\in \mathcal{S}'(\C^n):\|\widetilde{\mathcal{M}}_{\sigma}f\|_p<\infty\}$. Define $\|f\|_{h^p_{\sigma}}=\|\widetilde{\mathcal{M}}_{\sigma}f\|_p$. Goldberg \cite{Goldberg-1979} proved the following characterisation of local Hardy spaces ${h^p_{\sigma}(\mathbb{C}^n)}$, $0 <p\leq 1$. 
		
		\begin{theorem}\label{thm:Goldberg}
			For every tempered distribution $f$ on $\mathbb{C}^{n}$ and a fixed $\sigma>0$, the following conditions are equivalent.
			\begin{enumerate}
				\item $\widetilde{\mathcal{M}}_{\sigma}f \in L^p(\mathbb{C}^n)$
				\item $f(z)= \sum_{j} c_{j}f_{j}, \, \text{where}\, \supp f_{j}\subset Q(z_{j},r_{j})$ $$\|f_{j}\|_{\infty} \leq \left(r_{j}\right)^{-2n/p}$$
				and
				$$\, \int f_{j}(z) z^{\alpha} \bar{z}^{\beta} \, dz =0,$$ 
				$ \text{for} \, \text{all} \,\,  |\alpha|+|\beta| \leq \mathcal{N}\, \text{with}~ \, \mathcal{N} \geq \lfloor 2n(1/p-1) \rfloor , \, \text{whenever} \,\,\, r_{j}<\sigma$ and $\sum_{j}|c_j|^p<\infty.$
			\end{enumerate}
			
			Moreover $\|f\|_{h^p_{\sigma}}\simeq_{p,n} \inf \left\{ \left(\sum_{j} |c_{j}|^p \right)^{1/p}: f= \sum_{j}c_{j}f_{j}, f_j's~\text{are as in $(2)$ above}~ \right\}$.
		\end{theorem}
		In the above theorem the constants involved in the norm equivalence are independent of $\sigma>0$.
		In \cite{Goldberg-1979} the above theorem is stated for $\sigma=1$ which in turn implies it for any $\sigma>0$ by dilation and normalisation. A simple computation gives $\widetilde{\mathcal{M}}_{\sigma}f(x)=\widetilde{\mathcal{M}}_{1}f_{\sigma}(\sigma^{-1}x)$, where $f_{\sigma}(x)=f(\sigma x)$. In fact from the last identity involving maximal functions we get that $f\in h^p_{\sigma}(\C^n)$ if and only if $\sigma^{2n/p}f_{\sigma}\in h^p_{1}(\C^n) $ with norm equality. Once we get atomic decomposition of $\sigma^{2n/p}f_{\sigma}$ in $h^p_{1}(\C^n)$ with norms equivalence we get the corresponding characterisation of $h^p_{\sigma}(\C^n)$ as stated in Theorem \ref{thm:Goldberg} for any $\sigma>0$.
		
		To prove our Theorem \ref{thm:charac-hardy-spaces}, we first prove the following results. Let us state two useful lemma first
		
		\begin{lemma}\label{lem:compare-maximal-function}
			Let $f$ be a function such that $\supp f \subset Q(z_{0}, \sigma)$. Then there is a positive constant $C(\sigma)$ depending on $\sigma$ but independent of $z_{0}$ such that 
			\begin{align*} 
				C(\sigma)^{-1}\|\mathcal{M}_{\sigma}f\|_{p} \leq \|f(\cdot)\omega(z_{0},\cdot)\|_{h^{p}_{\sigma}} \leq C(\sigma) \|\mathcal{M}_{\sigma}f\|_{p}
			\end{align*}
			for every $0<p\leq 1$.
		\end{lemma}	
		
		The constant $C(\sigma)$ above is a polynomial in $\sigma$ of degree atmost $2N$.
		Let us consider a partition of $\mathbb{C}^{n}$ into family of cubes $Q_{j}=Q(z_{j}, \sigma/2)$ and construct a $C^{\infty}$ partition of unity $\{ \zeta_{j}\}$ such that $\zeta_{j}$ is supported on $Q_{j}^{*}=Q(z_{j}, \sigma)$. The cubes $\{Q_j^*\}$ have finite intersection property and {$|\partial^{\alpha}_{z} \partial_{\bar{z}}^{\beta} \zeta_{j}(z)| \lesssim_{\alpha,\beta} \sigma^{-|\alpha|-|\beta|}$} for all $\alpha, \beta$.

		\begin{lemma}\label{lem:compar-maximal and atomic decomp}
			Let $f$ be such that $\mathcal{M}_{\sigma}f\in L^p(\mathbb{C}^n)$. Then $g_{j}(z)= f(z)\zeta_{j}(z)\omega(z_{j},z)$ is in ${h^{p}_{\sigma}(\mathbb{C}^n)}$ and $\|\mathcal{M}_{\sigma}f\|_{p}^p \simeq_{\sigma} \sum_{j} \|g_{j}\|_{h^p_{\sigma}}^p$.
		\end{lemma}
		When $p=1$, the {Lemma} \ref{lem:compare-maximal-function} and {Lemma} \ref{lem:compar-maximal and atomic decomp} was proved in \cite{Mauceri-Picardello-Ricci-Hardy-space-1981}. For $0<p<1$ essentially same idea works, we are presenting the proof in the appendix ( see section \ref{Sec:Appendix}) for the sake of completion.
		\begin{remark}\label{R1}
			Note that Lemma \ref{lem:compar-maximal and atomic decomp} implies the following: Given $f$ such that $\mathcal{M}_{\sigma}f\in L^p(\C^n)$, using partition of $\C^n$ above we can first decompose $f=\sum_{j}g_j\bar{\omega}(z_j,z)$. We know that $g_j(\cdot)\in h^p_{\sigma}(\C^n)$ for each $j$. We can write the atomic decomposition of $g_j(\cdot)$ in $h^p_{\sigma}(\C^n)$ using Theorem \ref{thm:Goldberg}. 
			Therefore we conclude that for every $f$ such that $\mathcal{M}_{\sigma}f\in L^p(\C^n)$, we can write $f=\sum_{j} \eta_{j} h_{j}$ in $\mathcal{S}'(\R^{2n})$, where
			\begin{align}\label{atom condition}
				\sum_{j} |\eta_{j}|^{p}\lesssim_{n,p} C(\sigma)\|\mathcal{M}_{\sigma}f\|^{p}_{p}, \, \supp h_{j} \subseteq Q(w_{j}, r_{j}) \, \text{and} \, \|h_{j}\|_{\infty}\leq (r_{j})^{-2n/p}, 
			\end{align}
			and 
			\begin{align}\label{cancellation condition}
				&\text{whenever}~ \, r_{j}< \sigma \,, \text{there} \, \text{exists}~\, \vartheta_{j} \, \text{such} \, \text{that} \,  \vartheta_j\in Q(w_j,2\sigma) \, \text{and}\\
				\nonumber & \int h_{j}(z) z^{\alpha} \bar{z}^{\beta} e^{\frac{i}{2}Im (\vartheta_{j}\cdot \bar{z})} \, dz=0, \,\,\, \text{for} \, \text{all} \,\,\, |\alpha|+|\beta|\leq N ,\, \, \text{where} \,  N\geq \mathcal{N}_{0}.
			\end{align}
			
			Also, conversely given a sequence $\{h_{j}\}$ of functions satisfying \eqref{atom condition} and \eqref{cancellation condition} and a sequence $\{\eta_{j}\}$ satisfying $\sum_{j} |\eta_{j}|^{p}<\infty$, the function $f(z)= \sum_{j} \eta_{j} h_{j}$ satisfies $\mathcal{M}_{\sigma}f\in L^p(\C^n)$ and $ \|\mathcal{M}_{\sigma}f\|^{p}_{p} \leq C(\sigma) \sum_{j} |\eta_{j}|^{p}$. 
		\end{remark}
		Let us fix $N=2\mathcal{N}_0$ in the Remark \ref{R1} from now on. The atoms $h_j$ in the above Remark \ref{R1}  satisfy the cancellation condition  \eqref{cancellation condition} w.r.t $\vartheta_j$, which need not be the center of the cube where $h_j$ is supported. We need to replace $\vartheta_j$ in the cancellation condition \eqref{cancellation condition} by the center of the cube on which the atom is supported to get the desired characterisation in Theorem \ref{thm:charac-hardy-spaces}.
		
		\begin{lemma}\label{lem:center-atomic-decom}
			Let $f$ be a function supported on $Q(z_{0}, r), r<\sigma$ such that $$\|f\|_{\infty} \leq r^{-2n/p}$$ and 
			\begin{align}\label{NC}
				& \int f(w) w^{\alpha} \bar{w}^{\beta} e^{\frac{i}{2}Im(\vartheta\cdot \bar{w})} \, dw = 0,\,\,\, \text{for} \, \text{all} \,\,\, |\alpha|+|\beta|\leq 2\mathcal{N}_0 ,\\\nonumber &\, \,~ \text{and} ~\, \text{for} ~\, \text{some} \,~ \vartheta \,~ \text{with} ~\, \vartheta\in Q(z_{0},2\sigma).
			\end{align}	
			Let $\sigma$ be small enough (depending on $n$, $p$). Then $f$ can be decomposed as $f=\sum_{j} \eta_{j} g_{j}$, where 
			\begin{enumerate}[(i)]
				\item $\sum_{j} |\eta_{j}|^{p} \leq C$,
				\item $\supp g_{j} \subseteq Q(z_{j}, r_{j})$, $\|g_{j}\|_{\infty} \leq r_{j}^{-2n/p}$,
				\item $\int g_{j}(w) w^{\alpha} \bar{w}^{\beta} e^{\frac{i}{2}Im(z_{j}\cdot \bar{w})} \, dw = 0,\, \text{for}~ \, \text{all}~ \, |\alpha|+|\beta|\leq \mathcal{N}_0  ,\, \, \, \text{whenever} \, \, r_{j}<\sigma$.
			\end{enumerate}
			The constant $C>0$ depends on $n$ and $p$.
		\end{lemma}
		\begin{proof}
			Let $L^2(Q)$ be the Hilbert space of square integrable functions on a cube $Q$ with the norm \begin{align*}
				\|f\|_{Q}^{2}= \frac{1}{|Q|}\int_{Q} |f(z)|^{2}\, dz .
			\end{align*}  Let $\mathcal{R}$ be the vector space of all polynomial of degree upto $\mathcal{N}_0 $. We equip the space $\mathcal{R}$ with the Hilbert space norm $\|\cdot\|_{Q}$ defined above. Let us denote the space $\{\mathcal{R},\|\cdot\|_{Q}\}$ by $\mathcal{R}_{Q}$. We fix an orthonormal basis $\{e_{1}, \ldots e_{J}\}$ of $\mathcal{R}_{Q}$, where $dim (\mathcal{R}_{Q})= J$.
			
			When $Q=Q(0,1)$, using equivalence of any two norms on a finite dimensional vector space we get 
			\begin{align*}
				\sup_{w\in Q(0,1)} |P(w)| \leq C \|P\|_{Q(0,1)}.
			\end{align*}
			Define $Tf(x)=f(rx+z)$. Note that $T$ is an isometry onto from $\mathcal{R}_{Q(z,r)}$ to $\mathcal{R}_{Q(0,1)}$. The map $T$ depends on $z,r$. 
			Therefore, using the map $T$ for a given $Q(z,r)$, we get the following norm equivalence for all $Q$  \begin{align*}
				\sup_{w\in Q} |P(w)| \leq C \|P\|_{Q},
			\end{align*}
			where the constant $C$ in the above inequality is independent of the choice of $Q$.
			The above inequality further implies that
			\begin{align}\label{eq:bound-polynomial}
				\sup_{z\in Q} |e_{k}(z)| \leq C
			\end{align}
			for all $1\leq k \leq J$.
			
			Let us now fix $Q=Q(z_0,r)$. Define $h_k(w)=e_{k}(w)e^{-\frac{i}{2}Im(z_{0}\cdot \bar{w})}$ for $k=1,2,..,J$. Note that $( h_j,h_l)_{Q}=(e_j, e_l)_{Q}$. Let  $\mathcal{H}_{Q}$ be the linear span of the elements $\{h_k\}_{k=1}^{J}$. Clearly, $\mathcal{H}_{Q}\subset L^2(Q)$. Since $\{h_k\}_{k=1}^{J}$ are linearly independent, they also form an orthonormal basis for the space $\mathcal{H}_{Q}$.
			
			Let us consider a projection operator $\Pi_{Q}$ from $L^2(Q)$ onto the space $\mathcal{H}_{Q}$ as follows 
			\begin{align}\label{eqn:projection}
				\left( \Pi_{Q}f\right)(z)= \sum_{k=1}^{J}(f, h_k)_{Q} ~~e_{k}(z) e^{-\frac{i}{2}Im(z_{0}\cdot \bar{z})}\chi_{Q}(z).
			\end{align}
			
			We can write \begin{align}\label{eq:projection}f(z)= a^{(1)}(z)+ b^{(1)}(z),\end{align} where
			
			\begin{align*}
				b^{(1)}(z)= \Pi_{Q}(f)(z), \quad a^{(1)}(z)= f(z)- \Pi_{Q}(f)(z).
			\end{align*}
			
			By the definition of the projection $ \Pi_{Q}$, it is clear that $\frac{1}{CJ+1} a^{(1)}(z)$ satisfies the conditions $(ii)$ and $(iii)$.
			
			For $b^{(1)}(z)$, using the estimate \eqref{eq:bound-polynomial}, we get 
			\begin{align}\label{eq:bound-b1}
				|b^{(1)}(z)| \leq C\sum_{k=1}^{J} \left| \frac{1}{|Q|}\int_{Q} f(w) \overline{e_{k}(w)} e^{\frac{i}{2}Im(z_{0}\cdot \bar{w})} \, dw \right|
			\end{align}
			for $z\in Q$.
			
			We shall consider each term of the sum on the right side of \eqref{eq:bound-b1} separately. 
			We would like to approximate each quantity in the sum by $(p,\sigma)$ atoms 
			by utilising the condition {\eqref{NC}} and the fact that $\vartheta\in Q(z_0,2\sigma)$. In fact it turns out that for $\sigma$ small enough (depending on $n$ and $p$) we will be able to achieve our target. 
			
			First observe that the Taylor series expansion (w.r.t. standard translation) of the function $u\rightarrow e^{\frac{i}{2}Im(u\cdot \overline{w-z_{0}})}$ about the point $\vartheta\in\C^n$ is given by 
			\begin{align}\label{eq:Taylor-series-expo}
				&e^{\frac{i}{2}Im(u\cdot \overline{w-z_{0}})}
				=  \sum_{|\alpha|+|\beta|\leq \mathcal{N}_0} \frac{1}{\alpha! \beta!} \left(\frac{i}{2}\right)^{|\alpha|+|\beta|} (-1)^{|\alpha|} e^{\frac{i}{2}Im(\vartheta\cdot \overline{w-z_{0}})} \times\\\nonumber & \left( w-z_{0}\right)_{2}^{\alpha} \left( w-z_{0}\right)_{1}^{\beta} \left(u-\vartheta \right)_{1}^{\alpha} \left(u-\vartheta \right)_{2}^{\beta}+ (\mathcal{N}_0+1) \sum_{|\alpha|+|\beta|= \mathcal{N}_0+1}\frac{1}{\alpha! \beta!} \left(\frac{i}{2}\right)^{|\alpha|+|\beta|} \times \\
				\nonumber  & (-1)^{|\alpha|} \int_{0}^{1} (1-s)^{\mathcal{N}_0} e^{\frac{i}{2}Im((\vartheta + s (u-\vartheta))\cdot \overline{w-z_{0}})} \left( w-z_{0}\right)_{2}^{\alpha} \left( w-z_{0}\right)_{1}^{\beta} 
				\left(u-\vartheta \right)_{1}^{\alpha} \left(u-\vartheta \right)_{2}^{\beta} \, ds.
			\end{align}
			
			In the above expansion, we have used the following notations
			\begin{align*}
				w-z_{0}=&(w-z_{0})_{1} + i (w-z_{0})_{2}, \, (w-z_{0})_{1}, (w-z_{0})_{2}\in \mathbb{R}^{n},\\
				u-\vartheta=& (u-\vartheta)_{1}+ i (u-\vartheta)_{2}, \, (u-\vartheta)_{1}, (u-\vartheta)_{2} \in \mathbb{R}^{n}. 
			\end{align*}
			Let $I_k=\frac{1}{|Q|} \int_{Q} f(w) \overline{e_{k}(w)} e^{\frac{i}{2}Im(z_{0} \cdot \bar{w})} \, dw$.
			Using the Taylor series expansion \eqref{eq:Taylor-series-expo} at $u=z_0$ and the cancellation condition \eqref{NC} upto degree $2\mathcal{N}_0 $ on $f$ , we write
			\begin{align*}
				I_{k} = & \frac{1}{|Q|} \int_{Q} f(w) \overline{e_{k}(w)} e^{\frac{i}{2}Im(z_{0} \cdot \bar{w})} \, dw \\
				= & \frac{1}{|Q|} \int_{Q} f(w) \overline{e_{k}(w)} e^{\frac{i}{2}Im(z_{0}\cdot \overline{w-z_{0}})} \, dw
				- \frac{1}{|Q|} \sum_{|\alpha|+|\beta|\leq \mathcal{N}_0 } \frac{1}{\alpha! \beta!} \left(\frac{i}{2}\right)^{|\alpha|+|\beta|} (-1)^{|\alpha|} \\
				&   \left(z_{0}-\vartheta \right)_{1}^{\alpha} \left(z_{0}-\vartheta \right)_{2}^{\beta} \int_{Q} f(w) \overline{e_{k}(w)}  e^{\frac{i}{2}Im(\vartheta\cdot \overline{w-z_{0}})} \left( w-z_{0}\right)_{2}^{\alpha} \left( w-z_{0}\right)_{1}^{\beta} dw\\
				= &  (\mathcal{N}_0 +1) \sum_{|\alpha|+|\beta|= \mathcal{N}_0 +1} \frac{1}{\alpha! \beta!} \left(\frac{i}{2}\right)^{|\alpha|+|\beta|} (-1)^{|\alpha|} \left(z_0-\vartheta \right)_{1}^{\alpha} \left(z_0-\vartheta \right)_{2}^{\beta}\times \\ \frac{1}{|Q|}&\int_{Q} \int_{0}^{1} (1-s)^{\mathcal{N}_0} e^{\frac{i}{2}Im((\vartheta + s (z_0-\vartheta))\cdot \overline{w-z_{0}})}
				\left( w-z_{0}\right)_{2}^{\alpha} \left( w-z_{0}\right)_{1}^{\beta}	f(w) \overline{e_{k}(w)} \, ds \, dw.
			\end{align*} 
			Note that in condition \eqref{NC} we need $|\alpha| + |\beta|\leq 2\mathcal{N}_0 $ as $\{e_k\}$'s are polynomials of degree at most $\mathcal{N}_0 $. 
			Since $|w-z_{0}|\lesssim_n r$ and $|\vartheta-z_{0}|\lesssim_n 2 \sigma$, we get
			\begin{align}\label{eq:estimate-Ij}
				\left|I_{k}\right| \lesssim_{n,p} \frac{1}{|Q|} \int_{Q} |f(w)| \sigma^{\mathcal{N}_0 +1} \, r^{\mathcal{N}_0 +1} \, dw \leq C_{n,p} \sigma^{\mathcal{N}_0 +1} \, r^{\mathcal{N}_0 +1-2n/p}.
			\end{align}
			
			Therefore, the estimates \eqref{eq:bound-b1} and \eqref{eq:estimate-Ij} together imply
			\begin{align*}
				|b^{(1)}(z)|\leq C_{n,p,J} \sigma^{\mathcal{N}_0+1} \, r^{\mathcal{N}_0+1-2n/p}.
			\end{align*}
			
			If we choose $q=\frac{2n}{\frac{2n}{p}-\mathcal{N}_0-1}$, then $\|b^{(1)}\|_{q} \leq C_{n,p,J} \sigma^{\mathcal{N}_0+1}$. Recall that $2n(1/p-1)-1<\mathcal{N}_0=\lfloor 2n(1/p-1) \rfloor$ which further implies that $q>1$.
			
			Since $\supp b^{(1)} \subset Q(z_{0}, \sigma)$, by Holder's inequality and boundedness of $\widetilde{\mathcal{M}}_{\sigma}$ on $L^q(\C^n), q>1$ we get
			\begin{align*}
				\|\omega(z_0,\cdot)b^{(1)}(\cdot)\|_{h^p_{\sigma}}
				= & \|\widetilde{\mathcal{M}}_{\sigma}(\omega(z_0,\cdot) b^{(1)}(\cdot))\|_{p}\\
				\leq & C_{n,p,J} \sigma^{\mathcal{N}_0+1} \|\omega(z_0,\cdot)b^{(1)}(\cdot)\|_{q}\\
				\leq & C_{n,p,J} \sigma^{2(\mathcal{N}_0+1)}. 
			\end{align*}
			
			Now, we choose $\sigma$ sufficiently small such that $C_{n,p,J} \sigma^{2(\mathcal{N}_0+1)} < \frac{1}{2}$. Note that the choice of $\sigma$ only depends on $n,$ $p$ and is independent of $f$.
			
			Since $\omega(z_0,\cdot)b^{(1)} \in {h^{p}_{\sigma}(\mathbb{C}^n)}$, we can again write
			\begin{align*}
				b^{(1)}(z)= \sum_{j} \nu^{(1)}_{j} h^{(1)}_{j}(z)
			\end{align*}
			where $\sum_{j}|\nu^{(1)}_{j}|^{p} \lesssim_{n,p} \frac{1}{2^{p}}$ and the functions $h^{(1)}_{j}(z)$ satisfy the conditions \eqref{atom condition} and \eqref{cancellation condition} with $N=2\mathcal{N}_0$. In the rest of the proof whenever we refer to \eqref{cancellation condition} , we assume that $N=2\mathcal{N}_0$.
			
			We can now again decompose the functions $h^{(1)}_{j}$
			whose support is contained in a cube $Q_j=Q(z_{j}, r_{j})$ with $r_{j} <\sigma$ as done in \eqref{eq:projection}. Thus we can write 
			\begin{align*}
				b^{(1)}(z)= a^{(2)}(z)+ b^{(2)}(z).
			\end{align*}
			It is easy to see that $$\|\widetilde{\mathcal{M}}_{\sigma}(\omega(z_0,\cdot) b^{(2)}(\cdot))\|_{p}^p\leq \sum _{j}|\nu^{(1)}_j|^p\|\widetilde{\mathcal{M}}_{\sigma}(\omega(z_0,\cdot) \Pi_{Q_j}h_j^{(1)}(\cdot))\|_{p}^p.$$
			Note that in the above sum only those $Q_j$ occur for which the corresponding $r_j<\sigma$.
			Using the similar analysis as we did for $b^{(1)}$, we can show that 
			\begin{align*}
				\|\widetilde{\mathcal{M}}_{\sigma}(\omega(z_0,\cdot) \Pi_{Q_j}h_j^{(1)}(\cdot))\|_{p}^p\leq 1/2^p.
			\end{align*}
			This implies that \begin{align*}
				\|\omega(z_0,\cdot)b^{(2)}(\cdot)\|_{h^p_{\sigma}}
				= & \|\widetilde{\mathcal{M}}_{\sigma}(\omega(z_0,\cdot) b^{(2)}(\cdot))\|_{p}\\
				\leq & 1/4.
			\end{align*}
			Using the fact that $\omega(z_0,\cdot)b^{(2)}\in h^p_{\sigma}(\mathbb{C}^n)$ we can write $b^{(2)}=\sum_{j}\nu_j^{(2)}h^{(2)}_{j}$,
			where $h^{(2)}_{j}$ above satisfy the conditions \eqref{atom condition} and \eqref{cancellation condition} and $ \sum_{j} |\nu^{(2)}_{j}|^{p} \lesssim \frac{1}{4^{p}}$.
			
			Note that $a^{(2)}(z)= \sum_{j} \eta_{j}^{(2)} g_{j}^{(2)}(z)$ with $g_{j}^{(2)}$ satisfy $(ii)$ and $(iii)$ and $\sum_{j} |\eta_{j}^{(2)}|^{p}\lesssim 1$.
			
			By this iterative process, we get a sequence $\{a^{(k)}\}$ and $\{b^{(k)}\}$ such that 
			\begin{align*}
				b^{(k)}= a^{(k+1)} + b^{(k+1)},
			\end{align*}
			where $b^{(k+1)}(z)= \sum_{j} \nu^{(k+1)}_{j} h^{(k+1)}_{j}(z)$, with $h^{(k+1)}_{j}$ satisfy the conditions \eqref{atom condition}, \eqref{cancellation condition} and $ \sum_{j} |\nu^{(k+1)}_{j}|^{p} \lesssim \frac{1}{2^{(k+1)p}}$, and also $a^{(k+1)}(z)= \sum_{j} \eta_{j}^{(k+1)} g_{j}^{(k+1)}(z)$, with $g_{j}^{(k+1)}$ satisfy $(ii)$, $(iii)$ and $\sum_{j} |\eta_{j}^{(k+1)}|^{p}\lesssim 2^{-(k-1)p}$. 
			
			Therefore, we can write $f(z)= \sum_{k} a^{(k)}(z)$ in $\mathcal{S}'(\R^{2n})$, where $a^{(k)}\in H_{\CL, at,\sigma}^p(\mathbb{C}^n)$  and $\sum_{k}a^{(k)}$ converges in $H^{p}_{\CL,*, \sigma}(\mathbb{C}^n)$ with $\|f\|_{H^p_{\CL,at,\sigma}}\lesssim \|f\|_{H^p_{\CL,*,\sigma}}$. This gives the required atomic decomposition of $f$. 
			
			This completes the proof of Lemma \ref{lem:center-atomic-decom}.
			
		\end{proof}
		
		\subsection{Proof of Theorem \ref{thm:charac-hardy-spaces}}

		Before proving Theorem \ref{thm:charac-hardy-spaces}, let us first state a useful lemma we will be using throughout this paper.
		\begin{lemma}\label{lem:remainder}
			Let $g\in C^{\infty}(\R^{2n})$	 and $f$ be $(p,1)$ atom with $\supp{f}\subset Q(z_j,r)$, $r<1$. For $N\in \mathbb{N}$, we define 
			\begin{align*}
				\Phi_N(g,z,w)=\sum_{1\leq i_1,i_2,\ldots i_{N+1}\leq 2n } \frac{(-1)^{N+1}}{N!}& \int_{s=0}^{1}(1-s)^{N}\times\\  T_{i_1,i_2,\ldots i_{N+1}}(w-z_j){g}(z-z_j+s(z_j-w))& 
				e^{\frac{i}{2}s Im(z-z_j\cdot \overline{w-z_j})}	 ds.	\end{align*}
			Then $$f\times g(z)=e^{-\frac{i}{2}\text{Im}(z_j\cdot\overline{z})}\int_{\R^{2n}}f(w)\Phi_N(g,z,w)e^{\frac{i}{2}\text{Im}(z_j\cdot \overline{w})} dw.$$
		\end{lemma}
		\begin{proof}
			The proof follows from the Taylor's formula for twisted translation \eqref{eq:Taylor-formula-twisted} for $g$ and the fact that $f$ satisfies cancellation properties as in Definition \ref{Def:atom}.
			Using Taylor expansion \eqref{eq:Taylor-formula-twisted} for $g$ about $z-z_j$ and twisted translated by $w-z_j$ we get 
			\begin{align*}
				g(z-w)e^{\frac{i}{2}Im(z-z_j\cdot\overline{w-z_j})}=&= \sum_{k=1}^N\sum_{i_1,i_2,\ldots i_k=1}^{2n} \frac{(-1)^{k}}{k!}T_{i_1,i_2,\ldots i_k}(w-z_j)g(z-z_j) \\&+  \Phi_N(g,z,w),		\end{align*} 
			where \begin{align*}
				\Phi_N(g,z,w)=\sum_{1\leq i_1,i_2,\ldots i_{N+1}\leq 2n } \frac{(-1)^{N+1}}{N!}& \int_{s=0}^{1}(1-s)^{N}\times\\  T_{i_1,i_2,\ldots i_{N+1}}(w-z_j){g}(z-z_j+s(z_j-w))& 
				e^{\frac{i}{2}s Im(z-z_j\cdot \overline{w-z_j})}	 ds. \end{align*}	
			We know that $f$ is a $(p,1)$ atom and it satisfies the cancellation condition as in Definition \ref{Def:atom}. The Taylor expansion for $g(z-w)e^{\frac{i}{2}Im(z\cdot\overline{w})}$ as above and integrating it with $f$ proves the lemma.
			
		\end{proof}

		{\it Proof of Theorem \ref{thm:charac-hardy-spaces}}:		Let us denote the $\sigma$ in Lemma \ref{lem:center-atomic-decom} by $\sigma_0$. First of all note that $H^{p}_{\CL,*,\sigma_0}(\mathbb{C}^n) \subseteq H^{p}_{\CL,at, \sigma_0}(\mathbb{C}^n)$ as a consequence of Remark \ref{R1} combined with {Lemma \ref{lem:center-atomic-decom}}.	
		Let $f\in H^{p}_{\CL,*}(\mathbb{C}^n) $. Then, $f\in H^{p}_{\CL,*,\sigma}(\mathbb{C}^n)$ for any $\sigma<1$ by definition. In particular, for $\sigma_0$.  Again combining Remark \ref{R1} and Lemma \ref{lem:center-atomic-decom}, we can write $f=\sum_{j}\lambda_j a_j$ where $a_j's$ are ${(p, \sigma_0)}$-atoms and $\sum_{j}|\lambda_j|^p\lesssim_{n,p}\|f\|^p_{H^{p}_{\CL,*}}$. Let $a_j$ be a $(p,\sigma_0)$- atom supported in $Q(z_j,r_j)$ such that $\sigma_0\leq r_j<1$ and $\|a_j\|_{\infty}<(r_j)^{-2n/p}$. We will show that $a_j$ can be decomposed into a sum of two ${(p, 1)}$-atoms. Let $Q_j=Q(z_j, r_j)$ and $\Pi_{Q_j}$ be the projection operator defined in \eqref{eqn:projection}. 
		Write $a_j= \Pi_{Q_j}a_j +a_j-\Pi_{Q_j}a_j$.
		Using the definition of $\Pi_{Q_j}$ and  $L^{\infty}$ estimates on $a_j$, we get $$\|\Pi_{Q_j}a_j\|_{\infty}\lesssim_{J}\|a_j\|_{\infty}\lesssim_{J}r_{j}^{-2n/p}.$$ Since $r_j>\sigma_0$ and $\sigma_0>0$ is fixed which depends on $n,p$\textcolor{red}{,} we get $$\|\Pi_{Q_j}a_j\|_{\infty}\lesssim_{J}C(n,p).$$ Here $C(n,p)$ is independent of the choice of the atom $a_j$. Therefore $\Pi_{Q_j}a_j$ can be considered as an atom supported in $Q(z_j,1)$. 				
		By definition of $\Pi_{Q_j}$ it is clear that $a_j-\Pi_{Q_j}a_j$ satisfies the required cancellation condition and \begin{align*}\|a_j-\Pi_{Q_j}a_j\|_{\infty}\leq \|a_j\|_{\infty}+ \|\Pi_{Q_j}a_j\|_{\infty}\lesssim_{n,p}r_j^{-2n/p}.\end{align*}   We can treat $a_j-\Pi_{Q_j}a_j$ as supported in $Q(z_j,r_j)$ for a fixed $r_j<1$. 
		
		We will prove the other way {of} inclusion now. Let $f\in H^{p}_{\CL,at, \sigma}(\mathbb{C}^n)$.
		Then we can write $f=\sum_{j}c_{j}f_{j}$, where $f_{j}$'s are {$(p, \sigma)$}-atoms and $\sum_{j}|c_{j}|^{p}<\infty$. First, we shall show that for all $j$, 
		\begin{align} \label{eq:maximal-with-atom}
			\int \left(\mathcal{M}_{\sigma}f_{j}(z)\right)^{p}\, dz \leq C.	
		\end{align}
		
		Let us assume $\supp f \subset Q_{j}=Q(z_{j}, r_{j})$ and $\widetilde{Q}_{j}=4\sqrt{2n}Q_{j}$. We write
		\begin{align*}
			\int \left(\mathcal{M}_{\sigma}f_{j}(z)\right)^{p}\, dz =& \int_{\widetilde{Q}_{j}} \left(\mathcal{M}_{\sigma}f_{j}(z)\right)^{p}\, dz + \int_{\widetilde{Q}_{j}^c} \left(\mathcal{M}_{\sigma}f_{j}(z)\right)^{p}\, dz\\
			=& J_{1} +J_{2}.
		\end{align*}
		
		Since $|f_{j}|\leq r_{j}^{-2n/p}$, we have $\mathcal{M}_{\sigma}f_{j}(z) \lesssim r_{j}^{-2n/p}$. This implies
		\begin{align*}
			J_{1} \lesssim r_{j}^{-2n} |\widetilde{Q}_{j}| \leq C.
		\end{align*} 
		
		Now, we consider $J_{2}$. Let $z\notin \widetilde{Q}_{j}$. When $r_{j}\geq \sigma>t$, for $0<t<\sigma$, observe that $\supp (f_j \times\varphi_{t})$ is contained in $2\sqrt{2n}Q_j$. Therefore, $\mathcal{M}f(z)=0$, for $z \in \widetilde{Q}_j^c$ and the inequality \eqref{eq:maximal-with-atom} is obvious.	
		
		We now consider the case $r_{j}<\sigma$.
		
		Using the moment condition of the atom $f_{j}$ and the Lemma \ref{lem:remainder} for $\varphi_t$, we write
		\begin{align*}
			f_j \times\varphi_{t}(z)= \int_{\R^{2n}} \varphi_{t}(z-w) f_{j}(w) e^{\frac{i}{2}Im(z\cdot\bar{w})}\,dw\\
			=  e^{-\frac{i}{2}Im(z_j\cdot\overline{z})}\int_{\R^{2n}} f_{j}(w) \Phi_{\mathcal{N}_0}(\varphi_t,z,w)e^{\frac{i}{2}Im(z_j\cdot\overline{w})} \, dw  	
		\end{align*}

		When $0<t<\sigma$ and $\varphi\in C_c^{\infty}({\R^{2n}})$, one can check that $${\left|\tilde{X}^{\alpha} \tilde{Y}^{\beta} \varphi_{t}(z-z_j+s(z_j-w))\right|}\lesssim_{\sigma,\mathcal{N}_0}t^{-2n-\mathcal{N}_{0}-1}\|\varphi\|_{\mathcal{N}_{0}+1}$$
		for all $|\alpha|+|\beta|=\mathcal{N}_0+1$. Similar estimate is also true for other rearreangements of $\tilde{X}_j$ and $\tilde{Y}_j$.
		
		 Using the above estimate, we get
		\begin{align*}
			\left|\Phi_{\mathcal{N}_0}(\varphi_t,z,w)	 \right| \lesssim_{\sigma} \frac{r_{j}^{\mathcal{N}_{0}+1}}{t^{2n+\mathcal{N}_{0}+1}}. 
		\end{align*}

		Without loss of generality we can assume that $r_j<t$. Since $z\notin \widetilde{Q}_{j}$, $w\in Q_{j}$ and $z-w \in \supp \varphi_{t}$, we have $|z-z_{j}|\leq 2\sqrt{2n} t$. This implies 
		\begin{align*}
			\left|f_j\times \varphi_{t}(z)\right| \lesssim_{\sigma} r_{j}^{-2n/p} \frac{r_{j}^{2n+\mathcal{N}_{0}+1}}{|z-z_{j}|^{2n+\mathcal{N}_{0}+1}}.
		\end{align*}
		
		Therefore, for all $z\notin \widetilde{Q}_{j}$, we get
		\begin{align*}
			\mathcal{M}_{\sigma}f_{j}(z) \lesssim_{\sigma} r_{j}^{-2n/p} \frac{r_{j}^{2n+\mathcal{N}_{0}+1}}{|z-z_{j}|^{2n+\mathcal{N}_{0}+1}}.
		\end{align*}
		
		Since $(2n+\mathcal{N}_{0}+1)p>2n$, we get
		\begin{align*}
			J_{2}\lesssim_{\sigma} r_{j}^{-2n} \int_{\widetilde{Q}_{j}^{c}}  \left(\frac{r_{j}^{2n+\mathcal{N}_{0}+1}}{|z-z_{j}|^{2n+\mathcal{N}_{0}+1}}\right)^{p} \, dz \lesssim_{\sigma} r_{j}^{2n}r_{j}^{-2n} = C_{\sigma}.
		\end{align*}
		This proves the estimate \eqref{eq:maximal-with-atom}. 
		
		Now, we want to show that $f\in H^{p}_{\mathcal{L},*, \sigma}(\mathbb{C}^n)$. It follows from the following observation,
		\begin{align*}
			\int \left(\mathcal{M}_{\sigma}f(z)\right)^{p} \, dz \lesssim \sum_{j}|c_{j}|^{p} \int \left(\mathcal{M}_{\sigma}f_{j}(z)\right)^{p} \, dz \lesssim \sum_{j}|c_{j}|^{p} <\infty.
		\end{align*}
		
		This proves  $H^{p}_{\CL,at, \sigma}(\mathbb{C}^n)\subseteq H^{p}_{\mathcal{L},*, \sigma}(\mathbb{C}^n)$ for any $\sigma>0$, in particular for $\sigma=1$.

		\subsection{Atomic decomposition of Hardy space corresponding to the heat semigroup \texorpdfstring{$e^{-t^2\CL}$}{}}
		In this subsection, our main goal is to prove Theorem \ref{MainThmHardy}. 
		Let us recall that  $$H^{p}_{\mathcal{L}}
		(\mathbb{C}^n)= \{ f : M_{\CL}^{\text{heat}}f(z)= \sup_{0<t<\infty}\left| e^{-t^2\mathcal{L}}f(z)\right| \in L^p(\mathbb{C}^n)\}.$$	
		
		Note that for notational convenience we have used the heat operator $e^{-t^2\mathcal{L}}$ in the definition of $H^p_{\mathcal{L}}(\C^n)$.
		
		We will show that
		\begin{theorem}\label{thm:heatmaximal and twisted maximal hardy space}
			For any $0<p<\infty$,  there exist $c,c'>0$ independent of $f$ such that $$c\|f\|_{H^{p}_{\mathcal{L}}
			}\leq \|f\|_{H^{p}_{\CL,*}}\leq c'\|f\|_{H^{p}_{\mathcal{L}}}.$$ 	
		\end{theorem}

		As a consequence of Theorem \ref{thm:charac-hardy-spaces} and \ref{thm:heatmaximal and twisted maximal hardy space}, every element in $H^p_\mathcal{L}(\C^n)$ has an atomic decomposition with the atoms satisfying conditions in the Definition \ref{Def:atom}.
		Before proving the above theorem let us first prove an easy lemma.
		\begin{lemma}\label{lem:hardy-implies-heat-maximal}
			For $0<p<\infty$, $H^{p}_{\CL,*}(\mathbb{C}^n)\subseteq H^{p}_{\mathcal{L}}(\mathbb{C}^n)$.
		\end{lemma}
		
		\begin{proof}
			It suffices to show that for any $(p, 1)$-atom $f$, 
			\begin{align*}
				\|M_{\CL}^{\text{heat}}f\|_{p}\leq C
			\end{align*}
			where $C$ is a constant uniform in $f$.
			We know that $e^{-t^2\CL}f(z)=f\times p_{t^2}(z)$ where $p_{t^2}(z)=(4\pi)^{-n}(\sinh t^2)^{-n}e^{-\frac{1}{4}(\coth t^2)|z|^2}$ (see \cite[page 37]{Thangavelu-yellow-book}). 
			Let $f$ be a $(p, 1)$-atom such that $\supp f \subseteq Q(z_{0}, s)$. Since $B(z_0,\frac{s}{2})\subset Q(z_0,s)\subset B(z_0,\frac{s\sqrt{2n}}{2})$, without loss of generality, we can assume $\supp f \subseteq B(z_{0}, r)$ with $r=\frac{s\sqrt{2n}}{2}$. First, we write
			\begin{align*}
				&\int \left|M_{\CL}^{\text{heat}}f(z) \right|^{p} \, dz \\
				= & \int_{|z-z_{0}| \leq 2r} \left|M_{\CL}^{\text{heat}}f(z) \right|^{p}\, dz + \int_{|z-z_{0}| > 2r} \left|M_{\CL}^{\text{heat}}f(z) \right|^{p}\, dz \\
				= & \mathcal{I}_{1}+ \mathcal{I}_{2}.
			\end{align*}
			
			{We know that $M_{\CL}^{\text{heat}}$ is bounded on $L^2(\mathbb{C}^n)$.  Therefore, after applying Holder's inequality in $\mathcal{I}_1$ and using $L^2$ boundedness of $M_{\CL}^{\text{heat}}$, we get}
			\begin{align*}
				\mathcal{I}_{1}= & \int_{|z-z_{0}| \leq 2r} \left|M_{\CL}^{\text{heat}}f(z) \right|^{p}\, dz\\
				\lesssim_{n,p} & r^{2n(1-p/2)} \|f\|^{p}_{2}.
			\end{align*}
			The $L^{\infty}$ bound on $f$ gives $\mathcal{I}_{1}\lesssim_{n,p} 1$.

			For estimating $\mathcal{I}_{2}$, we write
			\begin{align*}
				\mathcal{I}_{2}= \sum_{k\geq 1}^{\infty} \int_{2^k r < |z-z_{0}|\leq 2^{k+1}r}  \left|M_{\CL}^{\text{heat}}f(z) \right|^{p}\, dz.
			\end{align*}
			
			For all $z, w$ such that $2^k r < |z-z_{0}| \leq 2^{k+1}r$ and $|w-z_{0}| <r$, we have that
			\begin{align*}
				|z-w|\geq r2^{k-1}.
			\end{align*}
			
			Let us first assume $r\geq 1$. Observe that using the fact that $\|f\|_{\infty}\leq r^{-2n/p}$, for $2^k r < |z-z_{0}| \leq 2^{k+1}r$, we have
			\begin{align*}
				\left| f\times p_{t^2}(z)\right|= & \left| \int p_{t^2}(z-w) f(w) e^{\frac{i}{2}Im(z\cdot \bar{w})} \, dw\right| \\
				\lesssim & r^{-2n/p}  \int \left|p_{t^2}(z-w)\right| \, dw\\
				\lesssim & r^{-2n/p} e^{-\frac{(2^k r)^2}{32}}  \int t^{-2n} e^{-\frac{|z-w|^2}{8t^2}} \, dw \\
				\lesssim & r^{-2n/p} e^{-\frac{(2^k r)^2}{32}}.
			\end{align*}
			In the above inequalities, we have used that $\coth \mu\geq 1$ for all $\mu\geq 0$.
			
			For $r\geq 1$, we get
			\begin{align*}
				\mathcal{I}_{2} \lesssim & \sum_{k\geq 1}  \int_{2^k r < |z-z_{0}|\leq 2^{k+1}r} r^{-2n}  e^{-\frac{(2^k r)^2p}{32}} \, dz \\
				\lesssim& \sum_{k\geq 1} (2^k r)^{2n} (2^k r)^{-\lambda}< \infty 
			\end{align*}
			if we choose $\lambda>2n$.
			
			Now, we estimate for $r< 1$. Using cancellation condition of the atom $f$, from Lemma \ref{lem:remainder} we write
			\begin{align*}
				f\times p_{t^2}(z)
				=e^{-\frac{i}{2}Im(z_0\cdot \bar{z})} \int f(w) \Phi_{\mathcal{N}_0}(p_{t^{2}},z,w)e^{\frac{1}{2}\text{Im}(z_0\cdot \overline{w})}dw.
			\end{align*}
			
			For simplicity let us estimate sum of terms of the type \begin{align*} \sum_{|\alpha|+|\beta|= \mathcal{N}_{0}+1} \frac{(-1)^{|\alpha|+|\beta|}}{(|\alpha|+|\beta|)!}(\text{Re}(w-z_{0}))^{\alpha} (\text{Im}(w-z_{0}))^{\beta} \\\times\int_{s=0}^{1}(1-s)^{\mathcal{N}_{0}}(\tilde{X})^{\alpha} (\tilde{Y})^{\beta}p_{t^2}(z-z_0 +sz_0- sw) e^{s\frac{i}{2}Im(z-z_{0}\cdot \bar{w-z_0})} ds.\end{align*}
			For other combinations of the terms as in Taylor's formula \eqref{eq:Taylor-formula-twisted} similar method works.
			
			The above expression can be dominated by
			\begin{align*}
				\sum_{|\alpha|+|\beta|= \mathcal{N}_{0}+1} \frac{1}{(|\alpha|+|\beta|)!}& \int_{s=0}^{1} \left|(\tilde{X})^{\alpha} (\tilde{Y})^{\beta}p_{t^2}(z-z_0 +sz_0- sw)\right| |(w-z_{0})^{\alpha}| |(w-z_{0})^{\beta}| \, ds\\
				\lesssim & r^{\mathcal{N}_{0}+1} \sup_{ |\alpha| + |\beta| = \mathcal{N}_0+1} \sup_{w\in B(z_{0},s),0<s<1} \left|(\tilde{X})^{\alpha} (\tilde{Y})^{\beta}p_{t^2}(z-w)\right|.
			\end{align*}
			
			One can easily check that
			\begin{align}\label{eq:heat-kernel-grad-esti}
				\left|(\tilde{X})^{\alpha} (\tilde{Y})^{\beta}p_{t^2}(z-w)\right| \leq C t^{-2n-\mathcal{N}_0-1} e^{-\gamma|z-w|^2 \coth t^2} 
			\end{align}
			for some constants $C$ and $0<\gamma<1/4$. Let $B=Q(z_{0},r)$.
			Since $|z-w|\geq 2^{k-1}r$ and using \eqref{eq:heat-kernel-grad-esti}, we get
			\begin{align*}
				|e^{-t^2\CL} f(z)| \lesssim & r^{\mathcal{N}_{0}+1} r^{-2n/p+2n} \int_{B} t^{-2n-\mathcal{N}_0-1} e^{-\gamma|z-w|^2 \coth t^2} dw\\
				\lesssim & r^{\mathcal{N}_{0}+1} r^{-2n/p+2n}t^{-2n} \int_{B}  |z-w|^{-\mathcal{N}_{0}-1} \left(\frac{|z-w|^2}{t^2}\right)^{\frac{\mathcal{N}_{0}+1}{2}}   e^{-\gamma\frac{|z-w|^2 \coth t^2}{2}} dw\\
				\lesssim & r^{\mathcal{N}_{0}+1} r^{-2n/p+2n} (2^k r)^{-(\mathcal{N}_{0}+1)} e^{-1/2 (2^k r)^2} t^{-2n}\int_{\mathbb{C}^n} \left(\frac{|z-w|^2}{t^2}\right)^{\frac{\mathcal{N}_{0}+1}{2}}   e^{-\gamma\frac{|z-w|^2}{2t^2}} dw\\
				\lesssim & r^{-2n/p+2n} 2^{-k(\mathcal{N}_{0}+1)} e^{-\gamma/2 (2^k r)^2}.
			\end{align*}
			
			Therefore, using the above estimate, we get 
			\begin{align*}
				\mathcal{I}_{2} \lesssim & \sum_{k\geq 1}^{\infty}  \int_{2^k r < |z-z_{0}|\leq 2^{k+1}r} r^{-2n+2np} 2^{-k(\mathcal{N}_{0}+1)p} e^{-\gamma\frac{p}{2} (2^k r)^2} \, dz\\
				\lesssim & \sum_{k\geq 1}^{\infty}   2^{2nk} r^{2np} 2^{-k(\mathcal{N}_{0}+1)p} (2^k r)^{-2np}\\
				\lesssim & \sum_{k\geq 1}^{\infty} 2^{-k(\mathcal{N}_{0}+1-2n(1/p-1))p} < \infty.
			\end{align*}
			the last inequality follows from the fact $\mathcal{N}_{0}+1>2n(1/p-1)$.
			
			This completes the proof of the Lemma \ref{lem:hardy-implies-heat-maximal}.
		\end{proof}
		
		\begin{theorem} \label{thm:heat-maximal-grand-maximal}
			For {$0<p<\infty$ and $0<\sigma\leq \infty$, $H^{p}_{\mathcal{L}}(\C^n)\subset H_{\CL,*,\sigma}^p(\C^n)$} and there exists $C>0$ such that $$\|f\|_{H_{\CL,*,\sigma}^p}\leq C \|f\|_{H^{p}_{\mathcal{L}}},$$ where $C$ depends only on $n,p$. In particular, for $\sigma=1$ as well.
		\end{theorem}
		
		
		To prove Theorem \ref{thm:heat-maximal-grand-maximal}, we will examine various maximal functions related to the heat operator $e^{-t^2\CL}$ and the relationships between them. The proof utilizes techniques developed by Fefferman-Stein in \cite{FS} and  Folland-Stein in \cite{Folland-Stein-book} to characterise the Hardy space for the case of sub-Laplacian on the Heisenberg group. The crucial thing is to be able to realise a distribution on $\C^n$ as a distribution on $\mathbb{H}^n$ via the lifting map $f\rightarrow f(z)e^{it}$. Utilizing this realisation and the Lie group structure of the Heisenberg group,  we will prove some essential estimates leading to Theorem \ref{thm:heat-maximal-grand-maximal}. Some of the proofs involved are standard and will be presented in the Appendix.

		Let us define non-tangential maximal function corresponding to $e^{-t^2\CL}$ as
		
		\begin{align*}
			& M^{*}_{\CL}f(z)= \sup_{|w|<t} \left| e^{-t^2\mathcal{L}}f(z-w)\right|.
		\end{align*}
		Note that for $w=0$ it reduces to ${M}^{heat}_{\mathcal{L}}$. so trivially we have that $${M}^{heat}_{\mathcal{L}}f(z)\leq M^{*}_{\CL}f(z).$$ In fact we will also show later the inequality in the $L^p$ norm in the reverse way in \eqref{eq:nontang-to-heat-maximal}. Before that we define another maximal function for our analysis
		\begin{align*}
			& M^{**}_{\CL,N} f(z)= \sup_{w\in \mathbb{C}^n,t>0} \left| e^{-t^2\mathcal{L}}f(z-w)\right| \left( 1+\frac{|w|}{t}\right)^{-N}
		\end{align*}
		
		for some $N>0$. We will make a choice of $N$ later.

		Observe that the following pointwise inequality follows just by using the definition of $M^{*}_{\CL}$ and $M^{**}_{\CL,N}$
		\begin{align*}
			M^{*}_{\CL}f(z) \leq 2^N M^{**}_{\CL,N} f(z).	
		\end{align*}
		
		We shall now state the reverse relation  {involving} $M^{*}_{\CL}$ and $M^{**}_{\CL,N}$ in $L^p$-norm. We will present the proof in the Appendix.
		\begin{lemma}\label{lem:comp-N-max-non-tang}
			If $M^{*}_{\CL}f(z) \in L^p(\mathbb{C}^n)$ and $N>2n/p$, then $M^{**}_{\CL,N} f(z) \in L^p(\mathbb{C}^n)$ with $\|M^{**}_{\CL,N} f\|_{p} \leq C_{N,p} \|M^{*}_{\CL} f\|_{p} $.
		\end{lemma}
		
		\begin{proof}[Proof of Theorem \ref{thm:heat-maximal-grand-maximal}]
			We first prove the following inequality involving non-tangential maximal function $M^{*}_{\CL}$ and grand maximal function $\mathcal{M}$ (defined in Section \ref{section:Hardy space}).
			\begin{align}\label{eq:grand-maxi-nontnag}
				\|\mathcal{M}f\|_{p} \leq C \|M^{*}_{\CL}f\|_{p} .
			\end{align}
			
			Before proving \eqref{eq:grand-maxi-nontnag}, let us first derive a useful relation between twisted convolution and convolution on Heisenberg group when restricted to particular type of functions.
			Define $F(z,u)= f(z) e^{iu}$. Let us choose $\psi(z,u)= \phi(z) \eta(u)$, where $\phi \in S(\mathbb{C}^n)$ and $\eta \in S(\mathbb{R})$ such that $\widehat{\eta}\equiv 1$ on $B(0,1)$. The standard dilation on Heisenberg group is defined by $\delta_{t} F(z,u)= t^{-2n-2}F(z/t, u/t^2)= t^{-2n}\phi(z/t) t^{-2}\eta(u/t^2)$. 
			Note that 
			\begin{align}\label{eq:approx-iden-trans}
				F\ast_{\mathbb{H}^n}\delta_{t}\psi(z,u)& = \int_{\mathbb{H}^n}  \delta_{t}\psi(w,v) F(z-w, u-v-\frac{1}{2}Im(z\cdot \bar{w})) \, dw \, dv\\
				\nonumber& = \int_{\mathbb{H}^n} \phi_{t}(w) t^{-2}\eta(v/t^2) f(z-w) e^{i(u-v-\frac{1}{2}Im(z\cdot \bar{w}))} \, dw \, dv\\
				\nonumber & = \int_{\mathbb{C}^n} \phi_{t}(w) f(z-w) \widehat{\eta}(t^2) e^{-\frac{i}{2}Im(z\cdot \bar{w}))} e^{iu}\, dw\\
				\nonumber & = f\times\phi_{t} (z),
			\end{align}
			
			for every $0<t<1$.
			Therefore $F\ast_{\mathbb{H}^n}\delta_{t}\psi(z,u)=f\times\phi_{t} (z)e^{iu}$ for $0<t<1$.
			Let $p_{t,\mathbb{H}^n}$ be the kernel of heat semigroup $e^{-t^2\CL_{\mathbb{H}^n}}$ associated with the Heisenberg sub-laplacian $\mathcal{L}_{\mathbb{H}^n}$. 
			
			From Theorem 4.9, page 117 of \cite{Folland-Stein-book}, we know that for $\psi \in  S(\mathbb{H}^n)$, there exists $\Theta^{s}, 0<s\leq 1$ such that 
			\begin{align*}
				& \psi = \int_{0}^{1}(p_{s,{\mathbb{H}^n}})  \ast_{\mathbb{H}^n} \Theta^{s}\, ds\\
				\text{and} \, & \int_{\mathbb{H}^n} \left( 1 + |z|\right)^{N} \left| \Theta^{s}(z,u)\right| dz du \leq C_{\psi, N} s^{N}. 
			\end{align*}
			
			The kernel of the heat semigroup for the Heisenberg sub-laplacian $\mathcal{L}_{\mathbb{H}^n}$ is given by the following  expression
			\begin{align*}
				p_{t,{\mathbb{H}^n}}(z,u)= c_{n} \int_{\mathbb{R}} e^{i\lambda u} p_{t^2,\lambda}(z) \, d\lambda,
			\end{align*} 
			where $p_{t^2,\lambda}$ is the kernel associated to the heat semigroup $e^{-t^2\mathcal{L}(\lambda)}$ and $\mathcal{L}(\lambda)$ is the scaled twisted Laplacian. Note that the above integral expression converges absolutely in $\lambda$. Indeed $p_{t^2,\lambda}(z)=(4\pi)^{-n}\frac{|\lambda|^n}{(\sinh |\lambda|t^2)^{n}}e^{-\frac{|\lambda|}{4}(\coth |\lambda|t^2) |z|^2}$ for $\lambda\neq 0$. See Theorem 2.8.1 of \cite{Thangavelu-book-uncertainty} for more details. 
			
			Using the above expression of $p_{t,{\mathbb{H}^n}}$, when we take $F(z,u)= f(z) e^{iu}$, we have $F\ast_{\mathbb{H}^n}(p_{t,{\mathbb{H}^n}})(z,u)=e^{-t^2\mathcal{L}}f(z) e^{iu}$. Using the convolution and dilation structure of the Heisenberg group we can check that $\delta_{t}\psi=\int_{0}^{1}(p_{st,{\mathbb{H}^n}})  \ast_{\mathbb{H}^n} \delta_{t}(\Theta^{s}) \, ds$.		Using this expression we have
			\begin{align}\label{eq:heat-kernel-trans}
				&\left|F\ast_{\mathbb{H}^n}\delta_{t}\psi(z,u)\right|\leq \int_0^1\left| F\ast_{\mathbb{H}^n}(p_{st,{\mathbb{H}^n}})\ast_{\mathbb{H}^n}\delta_{t}(\Theta^{s}) (z,u) \right|ds\\
				\nonumber =& \int_0^1\left| \int_{\mathbb{H}^n} \delta_{t}(\Theta^{s})(z-w, u-v-\frac{1}{2}Im(w\cdot \bar{z})) e^{-s^2t^2\mathcal{L}}f(w) e^{iv} \, dw \, dv \right|ds\\
				\nonumber\leq & \int_0^1\int_{\mathbb{H}^n} \left|\delta_{t}(\Theta^{s})(z-w, u-v-\frac{1}{2}Im(w\cdot \bar{z}))\right| (1+|z-w|/st)^{N} M^{**}_{\CL, N}f(z) \, dw \, dv \, ds\\
				\nonumber \lesssim & \int_0^1 s^{-N} \int_{\mathbb{H}^n} \left|\delta_{t}(\Theta^{s})(z-w, u-v-\frac{1}{2}Im(w\cdot \bar{z}))\right| (1+|z-w|/t)^{N} M^{**}_{\CL, N}f(z) \, dw \, dv \, ds\\
				\nonumber \lesssim &\int_0^1 M^{**}_{\CL, N}f(z) s^{-N}  \int_{\mathbb{H}^n} \left|(\Theta^{s})(w,v)\right| (1+|w|)^{N}  \, dw \, dv \, ds\\
				\nonumber \lesssim &  M^{**}_{\CL, N}f(z).
			\end{align}
			The estimates \eqref{eq:approx-iden-trans} and \eqref{eq:heat-kernel-trans} together imply
			\begin{align*}
				\mathcal{M}f(z)\leq C M^{**}_{\CL, N}f(z)
			\end{align*}
			for all $z$.
			
			Hence the above estimate together with Lemma \ref{lem:comp-N-max-non-tang} imply 
			\begin{align*}
				\|\mathcal{M}f\|_{p}\leq C \|M^{**}_{\CL, N}f\|_{p} \leq C \|M^{*}_{\CL}f\|_{p}
			\end{align*}
			for $N>2n/p$.

			Next, we shall show
			\begin{align}\label{eq:nontang-to-heat-maximal}
				\|M^{*}_{\CL}f\|_{p}\leq C \|M_{\CL}^{\text{heat}}f\|_{p}
			\end{align}
			to complete the proof of Theorem \ref{thm:heat-maximal-grand-maximal}.
			In \cite{FS}, Fefferman and Stein use the convolution structure to dominate non-tangential maximal function by the standard one. Here we will exploit the twisted convolution structure corresponding to $\CL$.
			For large $K$ and $\epsilon$ with $0<\epsilon \leq 1$, we define two new maximal functions 
			\begin{align*}
				& M^{\epsilon, K}_{\CL}f(z)= \sup_{|z-w|<t<\epsilon^{-1}} \left| e^{-t^2 \mathcal{L}}f(w)\right|\frac{t^{K}}{(\epsilon+t+ \epsilon |w|)^{K}}\\
				& M^{\epsilon, K}_{{\mathcal{L}}, \nabla}f(z)=\sup_{1\leq j \leq 2n} \sup_{|z-w|<t<\epsilon^{-1}}\left| tX_{j}e^{-t^2\mathcal{L}}f(w)\right|\frac{t^{K}}{(\epsilon+t+ \epsilon |w|)^{K}}
			\end{align*}
			where  ${X_{n+j}=Y_{j}}$, $1\leq j \leq n$. 
			
			To prove \eqref{eq:nontang-to-heat-maximal}, we shall prove that for any $p>q>0$,
			\begin{align}\label{eq:point-nontang-to-heat-maximal}
				M^{\epsilon, K}_{\CL}f(z)\leq C \left(M \left(M_{\CL}^{\text{heat}}f\right)^{q}(z)\right)^{1/q}	
			\end{align}
			for $z\in \mathbb{C}^n$, where $M$ is the Hardy-Littlewood maximal function. The boundedness of Hardy-Littlewood maximal function and the estimate \eqref{eq:point-nontang-to-heat-maximal} implies
			\begin{align*}
				\int_{\mathbb{C}^n} \left| M^{\epsilon, K}_{\CL}f(z)\right|^{p} \, dz \leq C \int_{\mathbb{C}^n} \left| M_{\CL}^{\text{heat}}f(z)\right|^{p} \, dz
			\end{align*}
			and taking limit $\epsilon \rightarrow 0$, we get \eqref{eq:nontang-to-heat-maximal}.
			
			Now, for proving \eqref{eq:point-nontang-to-heat-maximal}, we shall first show that
			\begin{align}\label{eq:grad-heat-to-nontang}
				\|M^{\epsilon, K}_{{\mathcal{L}}, \nabla}f\|_{p} \leq C \| M^{\epsilon, K}_{\CL}f \|_{p}.
			\end{align}

			For $1\leq j \leq n$, 
			\begin{align*}
				\mathcal{X}_{j} p_{t,{\mathbb{H}^n}}(z,u)= c_{n} \int_{\mathbb{R}} e^{i\lambda u} X_{j}(\lambda)p_{t^2,\lambda}(z) \, d\lambda
			\end{align*}	
			where {$X_{j}(\lambda)$ , $1\leq j\leq n$ are as given in the {Section \ref{Preliminaries}}}. We also have the inversion 
			\begin{align*}
				{X_{j}(\lambda)}p_{t^2,\lambda}(z)= c_{n}' \int_{\mathbb{R}} {\mathcal{X}_{j}} p_{t,{\mathbb{H}^n}}(z,u) e^{-i\lambda u} \, du
			\end{align*}
			for $\lambda\neq 0$.
			In particular, for $\lambda=1$,
			\begin{align}\label{eq:Xgradient}
				{X_{j}} p_{t^2}(z)= c_{n}' \int_{\mathbb{R}} {\mathcal{X}_{j}} p_{t,{\mathbb{H}^n}}(z,u) e^{-i u} \, du.
			\end{align} 	
			
			Similarly, for $1\leq j \leq n$, we get
			\begin{align}\label{eq:Ygradient}
				{Y_{j}}p_{t^2}(z)= c_{n}' \int_{\mathbb{R}} {\mathcal{Y}_{j}} p_{t,{\mathbb{H}^n}}(z,u) e^{-i u} \, du.
			\end{align}
			
			In order to prove \eqref{eq:grad-heat-to-nontang}, let us define 
			$\psi(z,u)= \mathcal{X}_{j} p_{1,{\mathbb{H}^n}}(z,u)$ for a fixed $j$. Using the fact that $\mathcal{X}_j\delta_tf(z,u)= t^{-1}\delta_t\left(\mathcal{X}_jf\right)(z,u)$, we get that $\delta_t\psi\left(z,u\right)=t{\mathcal{X}_{j}} p_{t^2,{\mathbb{H}^n}}(z,u)$. Using the identities \eqref{eq:Xgradient} and \eqref{eq:Ygradient}, a similar kind of analysis used to prove \eqref{eq:grand-maxi-nontnag} gives the following inequality
			\begin{align*}
				\|M^{\epsilon, K}_{{\mathcal{L}}, \nabla}f\|_{p} \leq C \| M^{\epsilon, K}_{\CL}f \|_{p}.	
			\end{align*} 
			where $C$ is independent of $\epsilon$.
			
			Now, let us consider the set 
			\begin{align*}
				\Gamma_{\tau,\epsilon} = \{ z: M^{\epsilon, K}_{{\mathcal{L}}, \nabla}f(z)\leq \tau M^{\epsilon, K}_{\CL}f(z)\}
			\end{align*}
			where $\tau$ is a positive number which will be chosen later.
			
			We claim that
			\begin{align} \label{eq:nontang-to-local-esti}
				\int_{\mathbb{C}^n} M^{\epsilon, K}_{\CL}f(z)^p \, dz \leq 2 \int_{\Gamma_{\tau,\epsilon}} M^{\epsilon, K}_{\CL}f(z)^p \, dz.
			\end{align} 
			
			Using the estimate \eqref{eq:grad-heat-to-nontang}, we get
			\begin{align*}
				\int_{\mathbb{C}^n\setminus\Gamma_{\tau,\epsilon}} M^{\epsilon, K}_{\CL}f(z)^p \, dz \leq \tau^{-p} \int_{\mathbb{C}^n\setminus\Gamma_{\tau,\epsilon}} M^{\epsilon, K}_{{\mathcal{L}}, \nabla}f(z)^p \, dz \leq C^{p} \tau^{-p} \int_{\mathbb{C}^n} M^{\epsilon, K}_{\CL}f(z)^p \, dz.
			\end{align*}
			
			We can choose $\tau^p \geq 2 C^{p}$ to get \eqref{eq:nontang-to-local-esti}. Therefore, we only need to estimate \eqref{eq:point-nontang-to-heat-maximal} for $z\in \Gamma_{\tau,\epsilon}$.  
			
			Let $z_{0}\in \Gamma_{\tau,\epsilon}$. By definition of $M^{\epsilon, K}_{\CL}f$ , there exists a $w \in \mathbb{C}^n$ and $t>0$ such that $|z_{0}-w|<t<\epsilon^{-1}$ and \begin{equation}\label{G1}\left|e^{-t^2\mathcal{L}}f(w)\right|\frac{t^{K}}{(\epsilon+t+ \epsilon |w|)^{K}}\geq \frac{1}{2} M^{\epsilon, K}_{\CL}f(z_0).\end{equation} Since $z_{0}\in \Gamma_{\tau,\epsilon}$, for each $1\leq j \leq 2n$, we get
			\begin{align*}
				\left| t{X_{j}}e^{-t^2\mathcal{L}}f(z)\right| \frac{t^{K}}{(\epsilon+t+ \epsilon |z|)^{K}} \lesssim \left|e^{-t^2\mathcal{L}}f(w)\right| \frac{t^{K}}{(\epsilon+t+ \epsilon |w|)^{K}}	
			\end{align*}  for all $|z-z_{0}|<t<\epsilon^{-1}$ and $w$ as in \eqref{G1}. For $z,w\in B(z_0,t)$ and $t<\epsilon^{-1}$ we have $$\frac{(\epsilon+t+ \epsilon |z|)^{K}}{(\epsilon+t+ \epsilon |w|)^{K}}\lesssim 1$$ for all $t>0$ with the upper bound independent of $z_0$. This implies \begin{equation}\label{IMP}
				\sup_{1\leq j \leq 2n} t \left| {X_{j}}e^{-t^2\mathcal{L}}f(z)\right| \leq C \left| e^{-t^2\mathcal{L}}f(w)\right| \end{equation} for all $z$ such that $|z-z_{0}|<t<\epsilon^{-1}$ and $w$ in \eqref{G1}.

			In what follows, $w$ always denotes the point satisfying the conditions in \eqref{G1}.
			Let us consider the set $\Lambda= \{ v : |v-z_{0}| <t, |v-w|<\frac{t}{4nC}\}$. Therefore, for $v\in \Lambda$, after applying mean value theorem for twisted translation for $p_{t^2}$ {(see \eqref{eq:Taylor-formula-twisted2})} 
			we get
			\begin{align*}e^{\frac{i}{2}Im (v-w)\cdot \overline{(u-w)}}p_{t^2}(v-w-(u-w))=p_{t^2}(w-u)\\+\sum_{1\leq j\leq 2n}\int_{0}^1 U_{j}(v-w)p_{t^2}(w-u+\theta(v-w))e^{-\frac{i}{2}\theta Im(v-w)\cdot\overline{(w-u)}}d\theta.\end{align*}
			On integrating the above expression with $e^{\frac{i}{2}Im (w\cdot \overline{u})}f(u)$,  and applying the reverse triangle inequality,  we get $$\left| e^{-t^2\mathcal{L}}f(v)\right| \geq  \left| e^{-t^2\mathcal{L}}f(w)\right|-2n|w-v|\sup_{1\leq j\leq 2n,|y-z_0|<t}|{X_{j}}e^{-t^2\mathcal{L}}f(y)|.$$
			
			From \eqref{IMP} it follows that $$\left| e^{-t^2\mathcal{L}}f(v)\right| \geq 1/2\left| e^{-t^2\mathcal{L}}f(w)\right|$$ where $v\in \Lambda$.
			Using the fact that $w$ is as in the inequality \eqref{G1}, we get \begin{equation}\label{eqn:last}\left| e^{-t^2\mathcal{L}}f(v)\right|\geq \frac{1}{2} \left|e^{-t^2\mathcal{L}}f(w)\right|\frac{t^{K}}{(\epsilon+t+ \epsilon |w|)^{K}}\geq \frac{1}{4} M^{\epsilon, K}_{\CL}f(z_0).\end{equation}
			
			Clearly  
			\begin{align*}
				M(M_{\CL}^{heat} f)^{q}(z_{0}) \geq & \frac{1}{|B(z_{0}, t)|} \int_{B(z_{0}, t)} M_{\CL}^{heat}f(v)^{q} \, dv \\
				\geq & \frac{1}{|B(z_{0}, t)|} \int_{B(z_{0}, t)} \left|e^{-t^2\mathcal{L}}f(v)\right|^{q} \, dv \\
			\end{align*}
			for any $t>0$.
			Now from the inequality \eqref{eqn:last}, for $z_{0}\in \Gamma_{\tau,\epsilon}$ and $v\in\Lambda$ we get,	
			
			\begin{align*}
				\frac{1}{|B(z_{0}, t)|} \int_{B(z_{0}, t)} \left|e^{-t^2\mathcal{L}}f(v)\right|^{q} \, dv &
				\geq  C_{1}  M^{\epsilon, K}_{\CL}f(z_{0})^q \frac{|\Lambda|}{|B(z_{0}, t)|}\\
				\geq & C_{2}  M^{\epsilon, K}_{\CL}f(z_{0})^q .
			\end{align*}	
			This proves the estimate \eqref{eq:point-nontang-to-heat-maximal}.	
			In the last inequality above we have used that $\frac{|\Lambda|}{|B(z_{0}, t)|}\geq C'$, where $C'>0$ is independent of $t$ and $z_0$ due to the standard dilation structure of $\R^{2n}$ and translation invariance of the Lebesgue measure. 
			This completes the proof of Theorem \ref{thm:heat-maximal-grand-maximal}.
		\end{proof}
		Let $\mathbb{H}^n_{\text{red}}:=\mathbb{H}^n/\{(0,2\pi \mathbb{Z})\}$ 	be the reduced Heisenberg group. Since $(0,2\pi \mathbb{Z})$ is a subgroup of the center of $\mathbb{H}^n$, we can also define $\mathbb{H}^n_{\text{red}}$ by the left quotient of $(0,2\pi \mathbb{Z})$.	
		We can identify $\mathbb{H}^n_{\text{red}}$ with the  set $\C^{n}\times S^1$. $\mathbb{H}^n_{\text{red}}$ is a nilpotent Lie group with a compact centre $0\times S^1$. 
			Define $Tf(z)=f(z)e^{it}$.
			
			\begin{remark}\label{rem:reduced}  We define  $$f*p_{t,\mathbb{H}^n}(z,s)= \int_{\mathbb{H}^n_{\text{red}}}f(w,\theta)p_t((w,\theta)^{-1}(z,s))dw d\theta  $$ 
				for $f\in \mathcal{S}(\mathbb{H}^n_{\text{red}})$. $f*p_{t,\mathbb{H}^n}$ defines a Schwartz class function on $\mathbb{H}^n_{\text{red}}$. We define a Hardy space on $\mathbb{H}^n_{\text{red}}$ as follows$$H^p({\mathbb{H}^n_{\text{red}}}):=\{f\in \mathcal{S}'(\mathbb{H}^n_{\text{red}}):\sup_{t>0}|f*p_{t,\mathbb{H}^n}|\in L^p(\mathbb{H}^n_{\text{red}})\}.$$
				Using a similar proof as in Theorem \ref{thm:heat-maximal-grand-maximal}, we can also show that $f\in H^p_{\mathcal{L}}(\C^n)$ if and only if $Tf\in H^p({\mathbb{H}^n_{\text{red}}})$.
				
			\end{remark}
			
			We conclude this section by presenting the proof of Theorem~\ref{MainThmHardy}. The result follows from the various statements established earlier in this section. However, for completeness, we provide the proof below.
			
			\begin{proof}[Proof of Theorem~\ref{MainThmHardy}]
				The equivalence between $i)$ and $iv)$ follows from Theorem~\ref{thm:charac-hardy-spaces}. Theorem~\ref{thm:heatmaximal and twisted maximal hardy space} establishes the equivalence between $ii)$ and $iv)$. Finally, the equivalence between $ii)$ and $iii)$ can be deduced from Remark~\ref{rem:reduced}.
				
				This completes the proof of Theorem~\ref{MainThmHardy}.
			\end{proof}

			\section{Sharp bounds for oscillatory multipliers on \texorpdfstring{$H^{p}_{\mathcal{L}}(\C^n)$}{}}\label{wave}

			{In this section, we prove Theorem \ref{thm:Main-result} concerning the boundedness of $\mathcal{L}^{-\delta/2} e^{ \pm it\sqrt{\mathcal{L}}},t>0$ from $H^{p}_{\mathcal{L}}(\C^n)$ to $L^p(\C^n)$ for $0<p\leq 1$. Our proof is inspired by the classical approach of estimating Fourier integral operators used by A. Miyachi \cite{Miyachi-wave} for the Euclidean Laplacian. However, the techniques utilized in \cite{Miyachi-wave} are not directly applicable to the case of the twisted Laplacian. For the Euclidean Laplacian, the solution to the wave equation can be expressed in terms of an integral operator with an explicit kernel. Unfortunately, such an expression does not exist for the twisted Laplacian.}
			
			{To address this, we employ method of subordination, which was also crucial in the work of D. M\"uller and A. Seeger \cite{Segeer-Muller-wave-APDE-2015} to write the wave operator into an integral involving Schr\"odinger operators up to a remainder term. For Schr\"odinger operators, we have an explicit kernel expression for the twisted Laplacian, which will be instrumental in obtaining the required estimates.}
			
			{One of the key ingredients in our analysis is the atomic space characterisation of $H^{p}_{\mathcal{L}}(\C^n),0<p\leq 1$, which has already been proved in the previous section. In view of the atomic space characterisation of $H^{p}_{\mathcal{L}}(\C^n),0<p\leq 1$, it suffices to prove the boundedness of $\mathcal{L}^{-\delta/2} e^{ \pm it\sqrt{\mathcal{L}}}$ for atoms only.} 
			
			{
				For smaller atoms, we must divide the frequency parameter into two parts based on the size of the atoms. In our proof, we also use a Taylor series expansion for twisted translation which is more natural in this context.
				\subsection{Proof of Theorem \ref{thm:Main-result}}
				\begin{proof}
					Let us define $\delta(n,p)=(2n-1)\left(1/p-1/2\right)$. It is enough to prove Theorem \ref{thm:Main-result} for $\delta= \delta(n,p)$ and  the operator $\mathcal{L}^{-\delta(n,p)/2} e^{ it\sqrt{\mathcal{L}}}$ at $t=1$. The proof of boundedness for $\mathcal{L}^{-\delta(n,p)/2} e^{\pm it\sqrt{\mathcal{L}}}$ for $t\neq 1$ follows in a similar fashion with a suitable modification of sub-ordination formula. See Remark \ref{remark:subordination-formula} below. Let  $$m(\sqrt{\mathcal{L}}):=\mathcal{L}^{-\delta(n,p)/2} e^{ i\sqrt{\mathcal{L}}}.$$
					
					Let $f$ be $(p,1)$ atom which is supported on a cube $Q=Q(z_0,s)$ of side length $s$ and centre at $z_0$. We know that $B(z_0,\frac{s}{2})\subset Q(z_0,s)\subset B(z_0,\frac{s\sqrt{2n}}{2})$. From now on we will work with $B(z_0,r)$, where $r=\frac{s\sqrt{2n}}{2}$. It suffices to prove that the following integral quantity
					\begin{align}
						\int_{\mathbb{C}^n} \left|m(\sqrt{\mathcal{L}})f(z)\right|^p dz 
					\end{align}	
					is uniformly bounded, independent $z_0\in\C^n, r>0$, i.e. the bound does not depend on the choice of the atom.
					Let us write \begin{align*}
						\int_{\mathbb{C}^n} \left|m(\sqrt{\mathcal{L}})f(z)\right|^p dz&=\int_{B(z_0,6 r)} \left|m(\sqrt{\mathcal{L}})f(z)\right|^{p} \, dz + \int_{B(z_0,6  r)^c} \left|m(\sqrt{\mathcal{L}})f(z)\right|^{p} \, dz. 
					\end{align*}	
					Using Holder's inequality in the first integral and boundedness of the operator $m(\sqrt{\mathcal{L}})$ on $L^2(\C^n)$, we get
					\begin{align*}
						\int_{B(z_0,6 r)} \left|m(\sqrt{\mathcal{L}})f(z)\right|^{p} \, dz
						\lesssim & \, r^{2n(1-p/2)} \left(\int_{\mathbb{C}^{n}} \left|m(\sqrt{\mathcal{L}})f(z)\right|^{2} dz \right)^{p/2} \\
						\lesssim  \, r^{2n(1-p/2)} \left(\int_{\mathbb{C}^{n}} \left|f(z)\right|^{2} \, dz \right)^{p/2}	
						\lesssim & \, r^{2n(1-p/2)} r^{-2n(1-p/2)} \lesssim 1.
					\end{align*}
					In the third inequality above, we have used that $\|f\|_{L^2} \lesssim_{n,p}r^{-2n(1/p-1/2)}$.	
					
					We will now show that		
					\begin{align} \label{eq:outside-compact-set}
						\int_{B(z_0,6r)^c} \left|m(\sqrt{\mathcal{L}}) f(z)\right|^p \, dz 	
					\end{align}
					has a uniform bound independent of the choice of the atom.
					
					Let us consider a partition of unity on $[1/2,\infty)$
					\begin{align*}
						\sum_{j=0}^{\infty} \phi (2^{-j} \lambda ) =1 \quad \text{for} \, \lambda \geq 1/2,
					\end{align*}
					where $\phi$ is a smooth function supported on $[1/4, 4]$ and $\phi\equiv 1$ on $[1/2,2]$. With this partition of unity, let us denote $m_j(\lambda)= m(\lambda)\phi(2^{-j}\lambda)$.
					Now we decompose our operator $m(\sqrt{\mathcal{L}})$ as follows 
					\begin{align}\label{eq:operator-decomposition}
						m(\sqrt{\mathcal{L}})= \sum_{j\geq 0} m_{j}(\sqrt{\mathcal{L}})
					\end{align}
					where $m_{j}(\sqrt{\mathcal{L}})= m(\sqrt{\mathcal{L}}) \phi(2^{-j}\sqrt{\mathcal{L}})$.
					

					We will use the following subordination formula from \cite[Proposition 4.1] {Segeer-Muller-wave-APDE-2015} 
					
					\begin{lemma}\label{lem:subordination-formula}
						Let $\tau \geq 1$ and $\chi\in {C_{c}^{\infty}}$ and supported in $[1/2,2]$. Then there exist $C^{\infty}_{c}$ functions $a_{\tau}(s)$ and $\Psi_{\tau}(s)$ supported for $s\sim 1 $ such that 
						\begin{align}\label{eq:subordination-MS}
							\chi(\tau^{-1}\sqrt{x})e^{i\sqrt{x}}= \sqrt{\tau} \int e^{i\tau/4s} a_{\tau}(s) e^{isx/\tau} ds + \Psi_{\tau}(\tau^{-2}x).
						\end{align}
						Furthermore $a_{\tau}$ and $\Psi_{\tau}$ are supported in $[1/16,4]$ and are linearly dependent on $\chi$. Also the following holds:
						$$\sup_{s}|\partial_s^{N_1}\partial_{\tau}^{N_2}a_{\tau}(s)|\leq C(K)\sum_{l=0}^K\|\chi^{(l)}\|_{\infty}~\text{for}~~N_1+N_2<\frac{K-1}{2}$$
						and $$\sup_{s}|\partial_s^{N_1}\partial_{\tau}^{N_2}\Psi_{\tau}(s)|\leq C(K,N_2)\tau^{N_1+1-K}\sum_{l=0}^K\|\chi^{(l)}\|_{\infty}~\text{for}~~N_1\leq K-2.$$
					\end{lemma}
					
					\begin{remark}\label{remark:subordination-formula}
						The above formula also holds for $\chi(\tau^{-1}\sqrt{x})e^{it\sqrt{x}}$ for $\tau t\geq 1$. The reader can refer to Proposition $2.2$ in \cite{Ricci-DAncona-wave-twisted-JFAA-2010}.	
					\end{remark}
					Putting $\tau=2^{j}$ and $\chi(x)=\phi(x)x^{-\delta(n,p)}$	 
					in the subordination formula above \eqref{eq:subordination-MS}, we get that 
					\begin{align}\label{eq:suborination-proof}
						2^{j\delta(n,p)}{m}_{j}(\sqrt{\mathcal{L}})=& 2^{j/2} \int e^{i\frac{2^j}{4s}} a_{2^j}(s) e^{is 2^{-j}\mathcal{L}} \, ds \, + \, \Psi_{2^j}(2^{-2j}\mathcal{L})\\
						\nonumber = & T_{j} + \Psi_{2^j}(2^{-2j}\mathcal{L}).
					\end{align}
					where 
					$$T_j= 2^{j/2} \int e^{i\frac{2^j}{4s}} a_{2^j}(s) e^{is 2^{-j}\mathcal{L}} \, ds .$$
					Therefore
					\begin{align*}
						m_j(\sqrt{\mathcal{L}})= 2^{-j\delta(n,p)} T_{j} + 2^{-j\delta(n,p)} \Psi_{2^j}(2^{-2j}\mathcal{L}).
					\end{align*}
					Note that $e^{is 2^{-j}\mathcal{L}}$ is the {Schr\"odinger} semigroup corresponding to the twisted Laplacian $\CL$. We know that for $f\in L^2(\C^n)$, $$e^{is 2^{-j}\mathcal{L}}f= f\times k_{s2^{-j}},$$ where the kernel $k_{s2^{-j}}(z)=\frac{(2\pi)^{-n}}{(\sin s2^{-j})^{n}} e^{-\frac{i}{4}\cot(s2^{-j})|z|^2}$. See (2.2.29) on page 37 of \cite{Thangavelu-yellow-book}.

					Using functional calculus corresponding to $\CL$, we have that $T_{j}f=f\times K_{j}$. 
					
					The above expression of $T_{j}$ implies that
					$$K_{j}(z) = (2\pi)^{-n}2^{j/2}  \int e^{i2^j/4s} e^{-\frac{i}{4}\cot(s2^{-j})|z|^2}\frac{a_{2^{j}}(s)}{(\sin s2^{-j})^{n}} \, ds.$$
					
					\begin{lemma}\label{lem:pointwise-subord-kernel-esti}
						There exist a large positive integer $N_0$, independent of $z$, such that for any $j\geq N_0$ and {$\alpha, \beta \in \mathbb{N}^n$}, we have 
						\begin{align}\label{eq:pointwise-subord-kernel-esti}
							\left|{{\tilde{X}}^{\alpha}{\tilde{Y}}^{\beta}K_{j}(z)}\right| \lesssim {2}^{j(2n+1)/2+j|\beta|+j|\alpha|} \left( 1+2^{j} |1-|z||\right)^{-N}
						\end{align}
						for all $N>0$ and for all $z\in \mathbb{C}^n$ such that $|1-|z||>0$.
					\end{lemma}
					Again, using the functional calculus corresponding to $\CL$, we can write $\Psi_{2^j}(2^{-2j}\mathcal{L})f = f\times K_{j,\Psi}$, for $f\in L^2(\mathbb{C}^n)$ and $j\geq 0$.
					\begin{lemma}\label{lem:weighted-sup-estimate}
						For any $N\geq 0$ and $\alpha, \beta\in \mathbb{N}^{n}$, we have 
						\begin{align}\label{eq:weighted-sup-estimate}
							\left|{{\tilde{X}}^{\alpha}{\tilde{Y}}^{\beta}K_{j,\Psi}(z)}\right| \lesssim 2^{j(2n+|\alpha|+|\beta|)} (1+2^j|z|)^{-N}
						\end{align}
						for all $z\in \mathbb{C}^n$, and $j\geq 0$.
					\end{lemma}
					Note that in Lemmas \ref{lem:pointwise-subord-kernel-esti} and \ref{lem:weighted-sup-estimate}, we present kernel estimates for vector fields of the form $\tilde{X}^{\alpha}\tilde{Y}^{\beta}$. However, as in the previous section, we require kernel estimates for all possible rearrangements of $\tilde{X}_j$ and $\tilde{Y}_j$. Since the proofs of \eqref{eq:pointwise-subord-kernel-esti} and \eqref{eq:weighted-sup-estimate} are analogous for other rearrangements of $\tilde{X}_j$ and $\tilde{Y}_j$, we write only for $\tilde{X}^{\alpha}\tilde{Y}^{\beta}$ to keep it simple.
					
					 We postpone the proof of the above two lemmas to the end of this section. Let us assume it for the time being and use it to prove Theorem \ref{thm:Main-result}. 
					
					We will divide our analysis into two parts: one for large frequencies and another for small frequencies. By large frequency, we refer to cases where $j \geq N_0 $, with $N_0$ being a large positive number from Lemma \eqref{lem:pointwise-subord-kernel-esti}. The remaining cases will be classified as small frequencies. 
					
					We will first prove our theorem for large frequencies. Additionally, we will separate our analysis into two cases: \( r \geq \frac{1}{3} \) and \( r < \frac{1}{3} \). 
					
					
					\vspace{0.5cm}
					
					\textbf{Case 1:} $r\geq \frac{1}{3}$. 
					{In this case, we do not require any cancellation condition on the atoms.}
					We first do our calculation for the operator $ T_{j}.$ Note that if $|z-z_0|>6r$ and $|w-z_0|\leq 2r$, we have $|z-w|-\frac{1}{6}|z-z_0|\geq 1$. Note that $\supp{f}\subset B(z_0,2r)$.
					We know that $T_jf=f\times K_j$.
					Now, using the estimate \eqref{eq:pointwise-subord-kernel-esti} for $\alpha=\beta=0$, we get 
					\begin{align} \label{eq:global-r-geat-1}
						&\int_{B(z_0,6r)^{c}} \left| T_{j}f(z)\right|^p dz \\
						\nonumber \leq & \int_{B(z_0,6r)^{c}} \int_{|w-z_0|\leq 2r}\left| \omega(z,w)K_{j}(z-w)f(w)\right|^p dw dz\\
						\nonumber \lesssim & \int_{B(z_0,6r)^{c}} \left(\int_{|w-z_0|\leq 2r} \frac{2^{j(2n+1)/2}2^{-jN_1}}{|1-|z-w||^{N_1}} \left|f(w)\right| dw \right)^p dz\\
						\nonumber \lesssim & 2^{j(2n+1)p/2}2^{-jN_1p} \int_{B(z_0,6r)^{c}} \frac{1}{|z-z_0|^{N_1p}} \|f\|_{L^1} dz\\
						\nonumber \lesssim & 2^{j(2n+1)p/2}2^{-jN_1p} r^{-2n(1-p)} r^{-N_1p+2n}.
					\end{align}
					
					Now summing over $j\geq N_0$ gives $$\sum_{j\geq N_0}\int_{B(z_0,6r)^{c}} \left| T_{j}f(z)\right|^p dz\lesssim 1$$
					if we choose $N_1>\max \{2n/p, n+1/2\}$ and $r\geq \frac{1}{3}$. 
					
					Let $\Psi_{2^j}(2^{-2j}\mathcal{L})f = f\times K_{j,\Psi}$. The boundedness of $\int_{B(z_0,6r)^{c}} \left| \Psi_{2^j}(2^{-2j}\mathcal{L})f(z)\right|^p dz$
					follows from the kernel estimates corresponding to the operator $\Psi_{2^j}(2^{-2j}\mathcal{L})$ in  Lemma \ref{lem:weighted-sup-estimate}.
					From the {Lemma} \ref{lem:weighted-sup-estimate} we get
					\begin{align}\label{eq:remainder-kernel-esti}
						\left|K_{j,\Psi}(z)\right| \lesssim 2^{jn}\left(1+2^j|z|\right)^{-N_2}
					\end{align}
					for all $N_2>0$. 
					
					Using the estimate \eqref{eq:remainder-kernel-esti}, we obtain
					\begin{align}\label{eq:global-r-geat-2}
						\int_{B(z_0,6r)^{c}} \left| \Psi_{2^j}(2^{-2j}\mathcal{L})f(z)\right|^p dz 	
						\leq & \int_{B(z_0,6r)^{c}}\left(\int_{{B(z_0,2r)}}\left| K_{j,\Psi}(z-w)f(w)\right| {dw} \right)^p dz\\
						\nonumber \lesssim & \int_{B(z_0,6r)^{c}} \left(\int_{{B(z_0,2r)}} \frac{2^{2jn}}{(1+2^j|z-w|)^{N_2}} \left|f(w)\right| {dw} \right)^p dz.
					\end{align}
					The above can be bounded by
					\begin{align*}
						\, 2^{2jnp} \int_{B(z_0,6r)^{c}} \frac{1}{(1+2^j|z-z_0|)^{N_2p}} \|f\|_{L^1}^{p} \, {dz}\\
						\lesssim  \, 2^{-2jn(1-p)} r^{-2n(1-p)} 2^{-j(N_2p-2n)} r^{-(N_2p-2n)}. 
					\end{align*}
					Again $$\sum_{j\geq N_0}\int_{B(z_0,6r)^{c}} \left| \Psi_{2^j}(2^{-2j}\mathcal{L})f(z)\right|^p dz\lesssim 1$$
					
					if we choose $N_2>2n/p$ and $r\geq \frac{1}{3}$.
					
					Using the estimates \eqref{eq:global-r-geat-1} and \eqref{eq:global-r-geat-2} and the fact $r\geq \frac{1}{3}$, we conclude that
					\begin{align*}
						& \int_{B(z_0,6r)^{c}} \left|\sum_{j\geq N_0} m_j(\sqrt{\mathcal{L}}) f(z)\right|^{p} dz 
						\leq  \sum_{j\geq N_0} \int_{B(z_0,6r)^{c}} \left|m_j(\sqrt{\mathcal{L}}) f(z)\right|^{p} dz \\
						\lesssim & \sum_{j\geq N_0} 2^{-j(2n-1)(1/p-1/2)p} \left(\int_{B(z_0,6r)^{c}} \left| T_{j}f(z)\right|^{p} dz + \int_{B(z_0,6r)^{c}}  \left|\Psi_{2^j}(2^{-2j}\mathcal{L}) f(z)\right|^{p} dz \right)
					\end{align*}
					$$\lesssim  1$$
					if we choose $N_1,N_2$ corresponding to the kernels of the operators $T_j$ and $\Psi_{2^j}$ as above.

					\vspace{0.5cm}
					\textbf{Case 2:} $r<\frac{1}{3}$. 
					
					{We write}		
					\begin{align*}
						\int_{B(z_0,6r)^{c}} \left|\sum_{j\geq N_0} m_j(\sqrt{\mathcal{L}}) f(z)\right|^p dz
						\lesssim & \sum_{j\geq N_0} \int_{B(z_0,6r)^{c}} \left|m_j(\sqrt{\mathcal{L}}) f(z)\right|^p dz \\
						=  \sum_{j\geq N_0: \, 2^{j}>1/r} \int_{B(z_0,6r)^{c}} \left|m_j(\sqrt{\mathcal{L}}) f(z)\right|^p dz +& \sum_{j\geq N_0: \,  2^{j}\leq 1/r} \int_{B(z_0,6r)^{c}} \left|m_j(\sqrt{\mathcal{L}}) f(z)\right|^p dz\\
						= I_1 + I_2.
					\end{align*}
					
					\subsection{When \texorpdfstring{$2^j>r^{-1}$}{}}
					
					\textbf{\underline{Estimation for $I_1$:}} 
					
					Using \eqref{eq:suborination-proof}, we first write
					\begin{align*}
						I_1 \leq & \sum_{j\geq N_0:\, 2^j > 1/r} 2^{-jp\delta(n,p)} \int_{{B(z_0,6r)^{c}}} \left| T_{j}f(z)\right|^p dz \\
						& + \sum_{j\geq N_0:\, 2^j > 1/r} 2^{-jp\delta(n,p)} \int_{{B(z_0,6r)^{c}}} \left| \Psi_{2^j}(2^{-2j}\mathcal{L})f(z)\right|^p dz\\
						=& I_{11} + I_{12}.
					\end{align*}

					First, we estimate $I_{11}$. 
					Define the following decomposition of {$\C^n$}: $$A_1=\{z:|z-z_0|\leq 1-3r,\}, A_2=\{z:1-3r<|z-z_0|\leq 1+3r\},$$
					$${A_3=\{z:1+3r<|z-z_0|\leq 2 \} ~ \text{and}~ A_4=\C^n\setminus\{A_1\cup A_2\cup A_3\}.}$$
					
					We further subdivide $I_{11}$ as follows
					
					\begin{align*}
						I_{11} =  \mathcal{I}_{1} + \mathcal{I}_{2}+\mathcal{I}_{3}+\mathcal{I}_{4},
					\end{align*}
					
					where 
					\begin{align*}
						\mathcal{I}_{i}= \sum_{j\geq N_0:\, 2^j > 1/r} 2^{-jp\delta(n,p)} \int_{A_i\cap  B(z_0,6r)^c} \left|T_{j} f(z)\right|^p dz
					\end{align*} 
					for $i=1,2,3,4.$
					
					Let us estimate $\mathcal{I}_{1}$ first. Using the estimate \eqref{eq:pointwise-subord-kernel-esti}, we write
					\begin{align*}
						\left|T_{j}f(z)\right|
						=  \left| \int {K_{j}(z-w)f(w) \omega(z,w) }\, dw  \right| \\
						\nonumber \lesssim  \int 2^{j(2n+1)/2} 2^{-jN}\left(|1-|z-w||\right)^{-N} |f(w)| \, dw.
					\end{align*}
					Notice that for $|z-z_0|\leq 1-3r$ and $|w-z_0|\leq 2r$, we have $1-|z-w|\geq \frac{1-|z-z_0|}{3}$. Therefore, we get that \begin{align}\label{pointwiseT_j}\left|T_{j}f(z)\right|\lesssim \, 2^{j(2n+1)/2} 2^{-jN}\left(|1-|z-z_0||\right)^{-N}  r^{-2n(1/p-1)}.\end{align}		
					
					In the last inequality, we used the fact $\supp f \subset B(z_0,r)$ and the estimate $\|f\|_{L^{\infty}} \leq r^{-2n/p}$.	
					Using the above pointwise estimate and integrating $\left|T_{j}f(z)\right|^p$ on the set $\{z:|z-z_0|\leq 1-3r\}$ gives 
					\begin{equation}\label{eq:pointwise-operator-subord-large-spectrum}
						\int_{|z-z_0|\leq 1-3r} \left| T_{j}f(z)\right|^{p} dz\lesssim
						2^{j(2n+1)p/2} 2^{-jNp}r^{-2n(1/p-1)p} r^{-Np+1}.\end{equation}
					
					Therefore, using the estimate \eqref{eq:pointwise-operator-subord-large-spectrum}, we get  
					\begin{align*}
						& \mathcal{I}_{1}\leq \sum_{j\geq N_0:\, 2^j > 1/r} 2^{-jp\delta(n,p)} \int_{|z-z_0|\leq 1-3r} \left| T_{j}f(z)\right|^{p} dz\\
						\lesssim & \, r^{-2n(1/p-1)p} r^{-Np+1} \sum_{j:\, 2^j > 1/r} 2^{-jp\delta(n,p)} 2^{j(2n+1)p/2} 2^{-jNp} .
					\end{align*}
					
					If we choose $N$ sufficiently large enough, then we have
					\begin{align*}
						& r^{-2n(1/p-1)p} r^{-Np+1} \sum_{j:\, 2^j > 1/r} 2^{-jp\delta(n,p)}  2^{j(2n+1)p/2} 2^{-jNp} \\
						\lesssim & 	r^{-2n(1/p-1)p} r^{-Np+1} r^{(2n-1)(1/p-1/2)p} r^{-(2n+1)p/2}r^{Np}= 1.
					\end{align*}
					
					Hence
					$$\mathcal{I}_{1} \lesssim 1.$$

					The estimate for the term $\mathcal{I}_{3}$ is similar to that of $\mathcal{I}_{1}$. We use the fact that $$|z-w|-1\geq \frac{1}{3}(|z-z_0|-1)$$ whenever $|z-z_0|\geq 1+3r$, $|w-z_0|\leq 2r$ and $r<\frac{1}{3}$. So we omit the details. 
					
					Now, we estimate  $\mathcal{I}_{4}$. To calculate $\mathcal{I}_{4}$, we use the pointwise estimate \eqref{pointwiseT_j} on $T_jf(z)$ and integrate on the set $\{z:|z-z_0|\geq 2\}$. Hence,
					\begin{align*}
						\mathcal{I}_{4} \lesssim_{n,p} &  \sum_{j\geq N_0:\, 2^j > 1/r} 2^{-jp\delta(n,p)} r^{-2n(1/p-1)p} 2^{j(2n+1)p/2} 2^{-jNp} \\
						\leq & \, r^{-2n(1/p-1)p} \sum_{j\geq N_0:\, 2^j > 1/r} 2^{-2jn(1/p-1)p}  \lesssim 1.
					\end{align*}
					{The last inequality is true if we choose $N>2n/p> 1/p.$}
					
					Finally, we estimate $\mathcal{I}_2$. We will use the $L^2$ boundedness of $T_j$ on $L^2(\C^n)$ uniformly in $j$. Note that $T_j=2^{j\delta(n,p)}{m_j}(\sqrt{\CL})-\Psi_{2^j}(2^{-2j\CL})$.
					The symbol  $$2^{j\delta(n,p)}{m_j}(\sqrt{\lambda})=\left(2^{-j}\sqrt{\lambda}\right)^{-\delta(n,p)}\phi(2^{-j}\sqrt{\lambda})e^{i\sqrt{\lambda}}$$ is uniformly bounded in $j$ and $\lambda$ as $1/2\leq 2^{-j}\sqrt{\lambda}\leq 2$ on the support of $\phi$. The symbol corresponding to $\Psi_{2^j}$ is also uniformly bounded in $\lambda$ and $j$. (See Lemma \ref{lem:subordination-formula}.) This implies that the $L^2$ boundedness of $T_j$ uniformly in $j$.		To estimate $\mathcal{I}_2$, we shall calculate the following 
					\begin{align}\label{eq:esti-near-singularity-I1-1}
						\sum_{j\geq N_0:\, 2^j > 1/r} 2^{-jp\delta(n,p)} \int_{1-3r \leq |z-z_0|\leq 1+3r} \left|T_jf(z)\right|^{p} dz.	
					\end{align}

					Using Holder's inequality, we have 
					\begin{align}\label{eq:esti-near-singularity-I1-2}
						& \int_{A_2} \left|T_jf(z)\right|^{p} dz\\
						\nonumber\lesssim & r^{1-p/2} \left( \int_{\mathbb{C}^n} \left|T_jf(z)\right|^{2} dz \right)^{p/2}\\
						\nonumber\lesssim & \, r^{1-p/2} \left( r^{2n(1-2/p)} \right)^{p/2},
					\end{align}
					where $A_2=\{z:1-3r<|z-z_0|\leq 1+3r\}$.
					Summing over $2^{j}>r^{-1}$, we get 
					\begin{align*}
						\sum_{j\geq N_0: 2^j > 1/r} 2^{-jp\delta(n,p)} \int_{A_2} \left|T_jf(z)\right|^{p} dz
						\lesssim \left(\sum_{j:\, 2^j > 1/r} 2^{-jp\delta(n,p)}\right) r^{1-p/2} r^{-2n(1-p/2)}\\
						\lesssim r^{1-p/2} r^{-2n(1-p/2)} r^{(2n-1)(1/p-1/2)p}& =1.
					\end{align*}
					
					This implies $$\mathcal{I}_2 \lesssim 1\, \quad \text{if} \,\,\, 0<p\leq 1.$$
					
					Now, we calculate $I_{12}$. Using the estimate \eqref{eq:weighted-sup-estimate}, we get
					\begin{align}\label{eq:remainder-kernel-esti-I1}
						\left|K_{j, \Psi}(z)\right| \lesssim_{N} 2^{2jn}\left(1+2^j|z|\right)^{-N}
					\end{align}
					for all $N>0$. 
					
					{Now, decomposing the range of the integration into the annulus and using Holder's inequality, we get
						\begin{align}\label{eq:subordi-reminder-estimate-I1}
							& \int_{{B(z_0, 6r)^c}} \left|\Psi_{2^j}(2^{-2j}\mathcal{L})f(z)\right|^{p} dz\\
							\nonumber= & \sum_{k\geq 0} \int_{62^kr<|z-z_0|\leq 62^{k+1}r} \left|\Psi_{2^j}(2^{-2j}\mathcal{L})f(z)\right|^{p} dz\\
							\nonumber\lesssim & \sum_{k\geq 0} \left(2^{k}r\right)^{2n(1-p)}\left(\int_{|z-z_0|> 62^{k}r} \left|\Psi_{2^j}(2^{-2j}\mathcal{L})f(z)\right| \, dz\right)^{p}.
						\end{align}

						Using the kernel estimate \eqref{eq:remainder-kernel-esti-I1}, we get
						\begin{align} \label{eq:subordi-reminder-esti-I1}
							& \int_{|z-z_0|> 6 2^{k}r} \left|\Psi_{2^j}(2^{-2j}\mathcal{L})f(z)\right| \, dz\\
							\nonumber \lesssim & \int_{|z-z_0|>62^{k}r} \int \frac{2^{2jn}}{(1+2^j|z-w|)^{N}} |f(w)| \, dz \, dw
						\end{align}
						for all integer $N>0$.
						
						Since support of $f$ is contained in $B(z_0,r)$, then $|w-z_0|\leq 2r$. This implies together with the fact $|z-z_0|\geq 6 2^kr$, we have
						{\begin{align*}
								|z-w| \geq |z-z_0|-|w-z_0| \geq 6\sqrt{2n}2^k r -2r > 2^kr. 
						\end{align*}}
						
						Therefore
						\begin{align}\label{eq:subordi-reminder-esti-I1-2}
							& \int_{|z-z_0|> 6 2^{k}r} \left|\Psi_{2^j}(2^{-2j}\mathcal{L})f(z)\right| \, dz\\
							\nonumber \lesssim & \int_{{B(z_0,r)}} |f(w)| \int_{|z|> 2^{k}r} \frac{2^{2jn}}{(1+2^j|z|)^{N}} \, dz \, dw\\
							\nonumber \lesssim & r^{-2n(1/p-1)} (2^j 2^k r)^{-N+2n},
						\end{align}
						where the last inequality is true if we choose $N>2n$.}
					
					Therefore, the estimates \eqref{eq:subordi-reminder-estimate-I1} and \eqref{eq:subordi-reminder-esti-I1-2} together imply
					\begin{align}\label{eq:subordi-reminder-esti-final-I1}
						& \sum_{j\geq N_0: 2^j > 1/r}\int_{|z-z_0|> 6r} \left|\Psi_{2^j}(2^{-2j}\mathcal{L})f(z)\right|^{p} dz\\
						\nonumber\lesssim & \sum_{k\geq 0} {2^{-k(Np-2n)} } \sum_{j: 2^j > 1/r} (2^j r)^{\left(-N+2n\right)p}<\infty.
					\end{align}
					The  sum is finite if we choose $N>2n/p\geq 2n$.
					Note that the estimate \eqref{eq:subordi-reminder-esti-final-I1} implies that
					\begin{align*}
						I_{12} \lesssim 1.
					\end{align*}

					\subsection{When \texorpdfstring{$2^j\leq r^{-1}$}{}}
					
					\vspace{0.2in} 
					
					\textbf{\underline{Estimation for $I_2$:} }
					
					Using \eqref{eq:suborination-proof} we write $I_2\leq  I_{21} + I_{22}$, where
					\begin{align*}
						&I_{21}  = \sum_{j\geq N_0:\, 2^j \leq 1/r} 2^{-jp\delta(n,p)} \int_{{B(z_0,6r)^c}} \left| T_{j}f(z)\right|^p dz \\
						& I_{22} = \sum_{j\geq N_0:\, 2^j \leq 1/r} 2^{-jp\delta(n,p)} \int_{{B(z_0,6r)^c}} \left| \Psi_{2^j}(2^{-2j}\mathcal{L})f(z)\right|^p dz.
					\end{align*}

					We first estimate $I_{21}$  and further subdivide as 
					\begin{align*}
						I_{21} =  \mathcal{J}_{1} + \mathcal{J}_{2}+\mathcal{J}_{3}+\mathcal{J}_{4},
					\end{align*}
					
					where 
					\begin{align*}
						\mathcal{J}_{i}= \sum_{j\geq N_0:\, 2^j \leq 1/r} 2^{-jp\delta(n,p)} \int_{A_i\cap B(z_0,6r)^{c}} \left|T_{j}f(z)\right|^p dz
					\end{align*} 
					for $i=1,2,3,4$.
					First, we estimate $\mathcal{J}_{1}$. Using the cancellation condition of the atom $f$ and the Taylor series expansion \eqref{eq:Taylor-formula-twisted} with respect to twisted translation for the kernel $K_{j}$ at $z-z_0$ and $w-z_0$ as in Lemma \ref{lem:remainder}, we get
				
				\begin{align*}
					\left|T_{j}f(z)\right|
					\lesssim &  \int_{\C^n}   |\Phi_{\mathcal{N}_0}(K_j,z,w)|  |f(w)|\, dw
				\end{align*}
				Recall that $\Phi_{\mathcal{N}_0}(K_j,z,w)$ can be dominated by sum of terms of the type $$\int_{s=0}^1\left|{\tilde{X}}^{\alpha}\tilde{Y}^{\beta}K_{j}(z-z_0+s(w-z_0))			(\text{Re}(w-z_0))^{\alpha} (\text{Im}({w}-{z}_{0}))^{\beta}\right|ds $$ for $|\alpha|+|\beta|=\mathcal{N}_0+1$ and other rearrangements.
				Using the fact that $\supp f\subset B(z_0,r)$ and the $L^{\infty}$ estimates on the atom $f$, we can dominate $T_jf$ as follows
				\begin{align*}
					\left|T_{j}f(z)\right|
					\lesssim_{\mathcal{N}_0} & \, r^{\mathcal{N}_0+1} r^{-2n(1/p-1)}  \sup_{{w \in B(z_0, r)}} \sup_{\alpha: |\alpha|  = \mathcal{N}_0+1} |{\tilde{X}_{i_1}}^{\alpha_1}{\tilde{X}_{i_2}}^{\alpha_2}\ldots\tilde{X}_{i_{2n}}^{\alpha_{2n}} K_{j}(z-w)|,		\end{align*}
				
				where supremum is over all rearrangements $(i_1,i_2,\ldots i_{2n})$ of $(1,2,\ldots,2n)$.
				
				
				Using the estimate \eqref{eq:pointwise-subord-kernel-esti}, we have
				\begin{align}\label{eq:pointwise-operator-subord}
					\left|T_{j}f(z)\right| \lesssim r^{\mathcal{N}_0+1} r^{-2n(1/p-1)} 2^{j(\mathcal{N}_0+1)} 2^{j(2n+1)/2} 2^{-jN}(|1-|z-z_0||)^{-N} .
				\end{align}
				Let $$I(r):=\int_{|z-z_0|\leq 1-3r} (1-|z-z_0|)^{-Np} \, dz.$$ Clearly, $I(r)\lesssim_{N,p}r^{-Np+1}$ for $0<r<1/3$.
				From \eqref{eq:pointwise-operator-subord-large-spectrum} we have
				\begin{align*}
					& \mathcal{J}_{1}= \sum_{j\geq N_0:\, 2^j \leq 1/r} 2^{-jp\delta(n,p)} \int_{|z-z_0|\leq 1-3r} \left| T_{j}f(z)\right|^{p} dz\\
					\lesssim & \, r^{(\mathcal{N}_0+1)p} r^{-2n(1/p-1)p}\sum_{j:\, 2^j \leq 1/r} 2^{-jp\delta(n,p)}2^{j(\mathcal{N}_0+1)p} 2^{j(2n+1)p/2} 2^{-jNp}~~ I(r).
				\end{align*}
				
				The last inequality is true if we choose $N>1/p$. $\mathcal{J}_1$ can be dominated by 
				\begin{align*}
					& r^{(\mathcal{N}_0+1)p-2n(1/p-1)p} r^{-Np+1} \sum_{j:\, 2^j \leq 1/r} 2^{-jp\delta(n,p)}  2^{j(\mathcal{N}_0+1)p} 2^{j(2n+1)p/2} 2^{-jNp}.
				\end{align*}
				
				The above is uniformly bounded in $0<r<1/3$ if we choose $$N<1/p + (\mathcal{N}_0+1 - 2n(1/p-1)).$$	 We leave the details to the reader. Hence, $\mathcal{J}_{1} \lesssim 1.$

				The estimate for the term $\mathcal{J}_{3}$ is similar to $\mathcal{J}_{1}$. So we omit the details. 
				
				Now we estimate  $\mathcal{J}_{4}$. We have 	
				\begin{align*}
					\mathcal{J}_{4}\lesssim \sum_{j\geq N_0:\, 2^j \leq 1/r}2^{-jp\delta(n,p)}\int_{A_4} |T_jf(z)|^p dz,
				\end{align*}
				where $A_4=\{z\in \C^n: |z-z_0|\geq 2\}$. Using pointwise estimates on $T_j$ from \eqref{eq:pointwise-operator-subord} and the fact that $\int_{A_4}(|1-|z-z_0||)^{-Np}\, dz$ is bounded, the above can be dominated by 
				\begin{align*}
					\mathcal{J}_{4}\lesssim  r^{(\mathcal{N}_0+1)p} r^{-2n(1/p-1)p}\sum_{j:\, 2^j \leq 1/r} 2^{-jp\delta(n,p)}  2^{j(\mathcal{N}_0+1)p} 2^{j(2n+1)p/2} 2^{-jNp} \\
					=~ r^{(\mathcal{N}_0+1)p} r^{-2n(1/p-1)p} \sum_{j:\, 2^j \leq 1/r} 2^{-2jn(1/p-1)p} 2^{j(\mathcal{N}_0+1)p} 2^{-j(N-1/p)p}\\
					\lesssim 1.
				\end{align*}
				{The above inequality is true if we choose $N> 1/p.$}
				
				
				{Next, we estimate $\mathcal{J}_2$. It is clear from \eqref{eq:operator-decomposition} and \eqref{eq:suborination-proof} that $T_j$ is a spectral multiplier operator whose multiplier function is supported on $[2^{j-2}, 2^{j+2}]$. Let $\chi_{1}$ be a smooth and compactly supported function on $[0, \infty)$ and $\chi_1= 1$ on $[1/4, 4]$. Then, we can write
					\begin{align*}
						T_j= \chi_1(2^{-2j}\mathcal{L}) T_j=T_j \chi_1(2^{-2j}\mathcal{L}).
					\end{align*}
				}
				
				Therefore, to estimate $\mathcal{J}_2$, we calculate the following 
				\begin{align}\label{eq:esti-near-singularity-1}
					\sum_{j\geq N_0:\, 2^j \leq 1/r} 2^{-jp\lambda(n,p)} \int_{A_2} \left|\chi_{1}(2^{-2j}\mathcal{L})T_jf(z)\right|^{p} dz,	
				\end{align}
				where $A_2=\{z\in \C^n:1-3r \leq |z-z_0|\leq 1+3r\}$.
				Writing $F_j(z)= \chi_{1}(2^{-2j}\CL)f(z)$ and using \eqref{eq:esti-near-singularity-I1-2}, we have 
				\begin{align}\label{eq:esti-near-singularity-2}
					\int_{A_2} \left|\chi_{1}(2^{-2j}\mathcal{L})T_jf(z)\right|^{p} dz \lesssim  r^{1-p/2} \left( \int_{\mathbb{C}^n} \left|F_j(z)\right|^{2} dz \right)^{p/2}.
				\end{align}
				
				We now need to estimate $\|\chi_{1}(2^{-2j}\mathcal{L})f\|_{L^2}$. We shall estimate this using the cancellation condition of the atom $f$. Let $G_j$ be the kernel of the operator $\chi_{1}(2^{-2j}\mathcal{L})$ such that $\chi_{1}(2^{-2j}\mathcal{L})f=f\times G_j$. By applying Lemma \ref{lem:remainder}, we get the following estimate
				
				\begin{align*}
					& \left|\chi_{1}(2^{-2j}\mathcal{L})f(z)\right|
					\lesssim \int_{\C^n}|\Phi_{\mathcal{N}_0}(G_j,z,w)||f(w)|dw.	
				\end{align*}

				Applying Mikowski's integral inequality, we have
				\begin{align*}
					&\| \chi_{1}(2^{-2j}\mathcal{L})f\|_{L^2}
					\lesssim \int_{\C^n}\left(\int_{\mathbb{C}^n} \left| \Phi_{\mathcal{N}_0}(G_j,z,w)\right|^{2} \, dz \right)^{1/2} |f(w)|\, dw.
				\end{align*}

				Using the Lemma \ref{lem:weighted-sup-estimate} to the kernel $G_{j}(z-w)$, we know that 
				\begin{align}\label{eq:nice-kernel-esti-pointwise}
					\left|{{\tilde{X}}^{\alpha}{\tilde{Y}}^{\beta}G_{j}(z-w)}\right| \lesssim \frac{2^{2jn}2^{j(|\alpha|+|\beta|)}}{(1+2^{j}|z-w|)^{N}}
				\end{align}
				for any large integer $N$ and similar estimates on other rearrangements of $\tilde{X}$ and $\tilde{Y}$.
				
				Therefore, using \eqref{eq:nice-kernel-esti-pointwise} with sufficiently large integer $N>0$ and following the similar method for getting estimate on $T_jf$ in \eqref{eq:pointwise-operator-subord-large-spectrum}, we get
				\begin{align}\label{eq:L2-esti-atom-cutoff}
					\| \chi_{1}(2^{-2j}\mathcal{L})f\|_{L^2} 
					\lesssim \, r^{\mathcal{N}_0+1} r^{-n(1/p-1)} M(\mathcal{N}_0)
				\end{align}
				where $$M(\mathcal{N}_0)=\sup_{w\in\C^n}\sup_{\alpha, \beta:|\alpha|+|\beta|=\mathcal{N}_0+1} \left( \int_{\mathbb{C}^n} \left|\frac{2^{2jn}2^{j(|\alpha|+|\beta|)}}{(1+2^{j}|z-w|)^{N}} \right|^2 dz \right)^{1/2}.$$
				Clearly, $M(\mathcal{N}_0)\lesssim 2^{j(\mathcal{N}_0+1)} 2^{jn}$.
				Using the estimates \eqref{eq:esti-near-singularity-2} and \eqref{eq:L2-esti-atom-cutoff}, we get 
				\begin{align*}
					& \sum_{j\geq N_0:\, 2^j \leq 1/r} 2^{-jp\delta(n,p)} \int_{A_2} \left|\chi_{1}(2^{-2j}\mathcal{L})T_jf(z)\right|^{p} dz\\
					\lesssim &  \sum_{j:\, 2^j \leq 1/r} 2^{-j(2n-1)(1/p-1/2)p} r^{1-p/2} \left( r^{\mathcal{N}_0+1} r^{-2n(1/p-1)} 2^{j(\mathcal{N}_0+1)} 2^{jn}\right)^{p}\\
					\lesssim & \, r^{1-p/2} r^{(\mathcal{N}_0+1-2n(1/p-1))p} r^{(2n-1)(1/p-1/2)p} r^{-(\mathcal{N}_0+1)p} r^{-np}=1.
				\end{align*}
				
				Therefore, we get $$\mathcal{J}_2 \lesssim 1$$
				and this completes the proof of $I_{21}$.
				
				Finally, we estimate $I_{22}$. Note that using the estimate \eqref{eq:weighted-sup-estimate}, we have
				\begin{align}\label{eq:remainder-kernel-esti-grad}
					\left|{{\tilde{X}}^{\alpha}{\tilde{Y}}^{\beta}K_{j, \Psi}(z)}\right| \lesssim 2^{2jn}2^{j(|\alpha|+|\beta|)}\left(1+2^j|z|\right)^{-N}
				\end{align}
				for all $N>0$ and $\alpha, \beta \in \mathbb{N}^n$. 
				{The proof of  $I_{22}$  follows a similar approach to that of $I_{21}$. Since we are estimating for those frequencies, for which $2^j \leq r^{-1}$, it is necessary to account for the higher-order cancellations of the atoms, just as we did for $I_{21}$. Additionally, we will incorporate the kernel estimates from \eqref{eq:remainder-kernel-esti-grad}. To keep our paper less technical, we will omit the detailed explanations here. Hence, this completes the proof of the theorem for large frequencies.} 
				
				For small frequencies, where $j<N_0$, it is sufficient to prove the estimate for a single $m_j$. We can express $m_{j}(\lambda)=\widetilde{m}_{j}(2^{-j}\lambda)$, where $\widetilde{m}_j(\lambda)=m(2^j\lambda)\phi(\lambda)$. This function  $\widetilde{m}_j(\lambda)$ is smooth and compactly supported on $[1/4,4]$. Thus, we can apply the Lemma \ref{lem:weighted-sup-estimate} to $\widetilde{m}_j$. The remaining estimate is similar to the remainder term $\Psi_{2^j}$.

				This completes the proof of the Theorem \ref{thm:Main-result}.

			\end{proof}

			\subsection{Proof of {Lemma} \ref{lem:pointwise-subord-kernel-esti}}
			
			Let 
			\begin{align*}
				K_{\tau}(z)= \sqrt{\tau} \int e^{i\tau/4s} e^{-i/4\cot(s/\tau)|z|^2} \frac{a_{\tau}(s)}{\sin (s/\tau)^{n}}  \, ds, \, \tau\geq 1, 
			\end{align*} 	
			where $a_{\tau}$ is given in Lemma \ref{lem:subordination-formula}.

			When $\tau=2^j$, $K_{\tau}$ is $K_j$ in the statement of Lemma \ref{lem:pointwise-subord-kernel-esti-tau}.
			The proof of the Lemma \ref{lem:pointwise-subord-kernel-esti} follows from the following lemma for the kernel $K_{\tau}$ {with} $\tau$ large. 
			
			\begin{lemma}\label{lem:pointwise-subord-kernel-esti-tau}
				For $\tau$ large, we have 
				\begin{align}\label{eq:pointwise-subord-kernel-esti-tau}
					\left|{\tilde{X}}^{\alpha}{\tilde{Y}}^{\beta}K_{\tau}(z)\right| \lesssim {\tau}^{(2n+1)/2+|\beta|+|\alpha|} \left( 1+\tau |1-|z||\right)^{-N}
				\end{align}
				for all $N>0$ and for all $z\in \mathbb{R}^n$ such that $|1-|z||>0$.
			\end{lemma}

			\begin{proof}
				First, we write the kernel $K_{\tau}$ as 
				\begin{align*}
					K_{\tau}(z)&= \sqrt{\tau} \int e^{i\tau/4s} e^{-i/4\cot(s/\tau)|z|^2} \frac{a_{\tau}(s)}{\sin (s/\tau)^{n}}  ds\\
					& = \sqrt{\tau} \int e^{i\Psi_{\tau}(s,z,u)/4} \frac{a_{\tau}(s)}{\sin (s/\tau)^{n}} ds
				\end{align*}
				where 
				\begin{align*}
					\Psi_{\tau}(s,z)= \frac{\tau}{s}-\cot\left(\frac{s}{\tau}\right)|z|^2.
				\end{align*}
				
				Let us first assume $|z|>1$. We estimate \eqref{eq:pointwise-subord-kernel-esti-tau} by using integration by parts many times. To do that, we shall first see that for $|z|>1$, the derivative of the phase function $\Psi_{\tau}(s,z)$ in the variable $s$ is non zero.
				
				Now, taking derivative of $\Psi_{\tau}(s,z)$ with respect to $s$, we get
				\begin{align*}
					\partial_{s}\Psi_{\tau}(s,z)=-\frac{\tau}{s^2}+ \frac{cosec^{2}(s/\tau)}{\tau} |z|^2 .
				\end{align*}
				
				Since $cosec^{2}(\frac{s}{\tau})\geq (\frac{\tau}{s})^2$ for $s\sim 1$ and $\tau\geq 1$, 
				\begin{align*}
					\partial_{s}\Psi_{\tau}(s,z) \geq -\frac{\tau}{s^2}+ \frac{\tau}{s^2} |z|^2 \geq \frac{1}{16} \tau(|z|^2-1)\geq \frac{1}{8} \tau (|z|-1) >0
				\end{align*}
				for  $|z|>1$.
				
				From Lemma \ref{lem:subordination-formula} we know that	$|\partial_{s}^{N} a_{\tau}(s)| \lesssim 1$ for any $N$ uniformly in $\tau$. A simple computation gives that $|\partial_s^j\left(\sin(\frac{s}{\tau})\right)^{-n}|\lesssim_{j} \tau^{-n}$ uniformly in $s\sim 1$ and $\tau$ large. Also note that $$\left|\frac{\partial_s^{l+1}\Psi_{\tau}(s,z)}{\partial_s\Psi_{\tau}(s,z)}\right|$$ is bounded for $|z|,\tau$ large (depending on $l\geq 1$) and $s\sim 1$.

				Now using the integration by parts $N$ many times for any $N\geq 0$, we get
				\begin{align}\label{eq:kernel-estimate-near-infinity}
					|K_{\tau}(z)| \lesssim \tau^{\frac{2n+1}{2}}\tau^{-N}(|z|-1)^{-N}
				\end{align} for $|z|\geq C(N)>1$.
				Here, we also use the fact that $\sin\left(\frac{s}{\tau}\right)\sim \frac{s}{\tau}$ for $\tau$ large.
				
				Now consider $z$  such that $|z|< C(N)$. We know that for small $x$,
				\begin{align}\label{eq:cot-expansion}
					\cot x = \frac{1}{x}+ O(x).
				\end{align}
				
				Using \eqref{eq:cot-expansion}, we write the phase function $\Psi_{\tau}(s,z)$ as follows 
				\begin{align*}
					\Psi_{\tau}(s,z)= \tau/s-\cot(s/\tau)|z|^2 = \tau/s (1-|z|^2) + O(s/\tau)|z|^2
				\end{align*}
				for $\tau$ large.
				Therefore, 
				\begin{align*}
					K_{\tau}(z)&= \sqrt{\tau} \int e^{i\Psi_{\tau}(s,z)/4} \frac{a_{\tau}(s)}{\sin (s/\tau)^{n}}\, ds\\
					& = \sqrt{\tau} \int e^{i\tau(1-|z|^2)/4s} e^{iO(s/\tau) |z|^2}\frac{a_{\tau}(s)}{\sin (s/\tau)^{n}}\, ds \\
					& = \sqrt{\tau} \int e^{i\tau(1-|z|^2)/4s} \upsilon_{\tau}(s,z)\,  ds	
				\end{align*}
				where $ \upsilon_{\tau}(s,z)= e^{iO(s/\tau) |z|^2}\frac{a_{\tau}(s)}{\sin (s/\tau)^{n}}.$ Since $|z| <C(N)$ , $s\sim 1$ and $\tau$ large, observe that $|\partial_{s}^{N}\upsilon_{\tau}(s,z)| \lesssim_N \tau^n $. Now, using the above fact and integration by parts, we get 
				\begin{align}\label{eq:kernel-estimate-near-zero}
					|K_{\tau}(z)| \lesssim \tau^{\frac{2n+1}{2}}\tau^{-N}\left(|1-|z||\right)^{-N}
				\end{align}
				for all $N>0$. Hence the estimates \eqref{eq:kernel-estimate-near-infinity}
				and \eqref{eq:kernel-estimate-near-zero} together imply the estimate \eqref{eq:pointwise-subord-kernel-esti-tau} for $|\alpha|=|\beta|=0$.

				Without loss of generality, we can assume that $\beta=0$ and  $\alpha=e_j$, where $1$ lies in the $j$-th cordinate. First, observe that 
				\begin{align} \label{eq:derivative-kernel-estimate}
					& \tilde{X}_j\left(K_{\tau}\right) (z) \\
					\nonumber = & -\frac{i}{2} \text{Im}(z_{j})  K_{\tau}(z) -\frac{i}{2} \text{Re}(z_j)  \sqrt{\tau} \int e^{i\Psi_{\tau}(s,z)/4} \cot(s/\tau) \frac{a_{\tau}(s)}{\sin (s/\tau)^{n}}\, ds.
				\end{align}
				
				Since $|\text{Im}(z_{j})|\leq |z|$ and $\tau$ is large, the kernel estimates \eqref{eq:kernel-estimate-near-infinity} and \eqref{eq:kernel-estimate-near-zero} of $K_{\tau}(z)$ imply 
				\begin{align} \label{eq:kernel-esti-derivative-pointwise-1}
					|\text{Im}(z_{j})  K_{\tau}(z)| \lesssim \tau^{\frac{2n+1}{2}}\tau^{-N}\left(|1-|z||\right)^{-N}
				\end{align}
				for all $N>0$. 
				
				Now, we estimate the second term of the right side of \eqref{eq:derivative-kernel-estimate}. Note that $\cot(s/\tau)\sim \tau/s$ for $s\sim 1 $ and $\tau$ large. Since $|\text{Re}(z_j)| \leq |z| $, and using the same technique that we used to estimate the term $K_{\tau}(z)$, we have
				\begin{align}\label{eq:kernel-esti-derivative-pointwise-2}
					& \left|\text{Re}(z_j)  \sqrt{\tau} \int e^{i\Psi_{\tau}(s,z)/4} \cot(s/\tau) \frac{a_{\tau}(s)}{\sin (s/\tau)^{n}}\, ds \right| \\
					\nonumber & \lesssim \tau^{\frac{2n+1}{2}+1}\tau^{-N}(|1-|z||)^{-N}
				\end{align}
				for all $N>0$.
				
				Hence, finally the estimates \eqref{eq:kernel-esti-derivative-pointwise-1} and \eqref{eq:kernel-esti-derivative-pointwise-2} imply that
				\begin{align*}
					\left| \tilde{X}_j K_{\tau}(z)\right| \lesssim \tau^{\frac{2n+1}{2}+1}\tau^{-N}(|1-|z||)^{-N}
				\end{align*}
				for all $N>0$. This completes the proof of the Lemma \ref{lem:pointwise-subord-kernel-esti-tau}. 
			\end{proof}

			\begin{proof}[Proof of Lemma \ref{lem:weighted-sup-estimate}]
				{Recall that $p_t(z)=(4\pi)^{-n}(\sinh t)^{-n}e^{-\frac{1}{4}(\coth t)|z|^2}$ is the kernel of the heat semigroup $e^{-t\mathcal{L}}$, we can check that there exists $c>0$ (depending on $|\alpha|$ and $|\beta|$) such that
					\begin{align}\label{eq:heat-kernel}
						\left|{\tilde{X}}^{\alpha}{\tilde{Y}}^{\beta} p_t(z)\right| \lesssim t^{-n-|\alpha|-|\beta|} e^{-c\frac{|z|^2}{t}}
					\end{align}
					for all $\alpha, \beta \in \mathbb{N}^n$ and $z\in \mathbb{C}^n$.}
				
				{If $\eta$ is a smooth and compactly supported on $[R/4, 4R]$, $R>0$, then the spectral multiplier operator $\eta(\sqrt{\mathcal{L}})$ can be written as an integral operator given by twisted convolution with its kernel as $K_{\eta(\sqrt{\mathcal{L}})}(z)$. Therefore by Lemma $4.3$ of \cite{Bui-Duong-2021}, we have 
					\begin{align*}
						(1+R|z-w|)^{N}\left|K_{\eta(\sqrt{\mathcal{L}})}(z-w)\right| \lesssim_{N} |B(z, R^{-1})|^{-1} \|\eta(R\cdot )\|_{W^{\infty}_{N+\epsilon}(\mathbb{R})} 
					\end{align*}
					for any $\epsilon >0$, $R>0$ and $N>0$, where $W^{\infty}_{N+\epsilon}$ is a Sobolev space of order $N+\epsilon$.}
				
				{Using techniques similar to those in Lemma 4.3 of \cite{Bui-Duong-2021}, along with the heat kernel estimate \eqref{eq:heat-kernel}, for all \(\alpha, \beta \in \mathbb{N}^n\), we obtain
					\begin{align}\label{eq:multiplier-gen-grad}
						(1+R|z-w|)^{N}\left|{\tilde{X}}^{\alpha}{\tilde{Y}}^{\beta}K_{\eta(\sqrt{\mathcal{L}})}(z-w)\right| \lesssim_{N} R^{|\alpha|+|\beta|}|B(z, R^{-1})|^{-1} \|\eta(R\cdot )\|_{W^{\infty}_{N+\epsilon}(\mathbb{R})} 
					\end{align}
					for any $\epsilon >0$, $R>0$ and $N>0$.}
				
				{If we take $\eta(s)= m(R^{-2}s^2)$ in \eqref{eq:multiplier-gen-grad}, we get
					\begin{align*}
						(1+R|z-w|)^{N}\left|{\tilde{X}}^{\alpha}{\tilde{Y}}^{\beta}K_{m(R^{-2}\mathcal{L})}(z-w)\right| &\lesssim_{N} R^{|\alpha|+|\beta|} |B(z, R^{-1})|^{-1} \|m \|_{W^{\infty}_{N+\epsilon}}\\ &\lesssim_{m,N} R^{2n+|\alpha|+|\beta|}.
					\end{align*}
				}
				{Now, taking $R=2^j$	and $m=\Psi_{2^j}$ in the above estimate, we get
					\begin{align}\label{ker-grad-psi}
						\left|{\tilde{X}}^{\alpha}{\tilde{Y}}^{\beta}K_{\Psi_{2^j}(2^{-2j}\mathcal{L})}(z-w)\right| \lesssim_{N} \frac{2^{j(2n+|\alpha|+|\beta|)} }{(1+2^j|z-w|)^{N}}
					\end{align}
					for all $\alpha, \beta \in \mathbb{N}^n$ and $N>0$.}
				
				In fact in the notation of Lemma \ref{lem:weighted-sup-estimate}, we have 
					\begin{align*}
						&K_{\Psi_{2^j}(2^{-2j}\mathcal{L})}(z- w)= K_{j,\Psi}(z-w).
					\end{align*}
				
				{Finally, using the above observations and the estimate \eqref{ker-grad-psi}, we get the desired estimates for the kernel $K_{j,\Psi}$.}
			\end{proof}
			
			
			\section{Appendix}\label{Sec:Appendix}
			In this section, we give the proofs of Lemma \ref{lem:compare-maximal-function}, Lemma \ref{lem:compar-maximal and atomic decomp} and  Lemma \ref{lem:comp-N-max-non-tang}.\\

			Proof of Lemma \ref{lem:compare-maximal-function}:	
			\begin{proof}
				As twisted translation commutes with twisted convolution, it suffices to take $z_{0}=0$. Let $\varphi \in \mathcal{S}_{N}$ for some fixed and large $N$. For $0<t<\sigma$, we write
				\begin{align*}
					f\ast\varphi_{t}(z) = &  \int \varphi_{t}(z-w)f(w)\, dw\\
					= & \int \varphi_{t}(z-w) e^{\frac{i}{2}Im(z\cdot \overline{z-w})} f(w) e^{\frac{i}{2}Im(z\cdot \bar{w})}\, dw\\
					= & f\times\left(\psi^{z}\right)_{t}(z)	
				\end{align*} 
				where $\psi^{z}(w)= \varphi(w) e^{\frac{i}{2}Im (z\cdot \overline{tw})}$.
				
				We know that $\supp \varphi \subset Q(0,1)$ and $\supp f \subset Q(0,\sigma)$. For $0<t<\sigma$, we get
				\begin{align*}
					f\ast\varphi_{t}(z) = 0 \quad \text{if}\, z\in Q(0,2\sigma)^c.	
				\end{align*}
				
				We can find a constant $C(\sigma)$ independent of $f$ such that $C(\sigma)^{-1}\psi^{z}\in S_{N}$ for all $z\in Q(0,2\sigma)^c$. Therefore 
				\begin{align*}
					\widetilde{\mathcal{M}}_{\sigma}f(z)\leq C(\sigma) \mathcal{M}_{\sigma}f(z)
				\end{align*}
				for all $z\in \C^n$.
				Similarly, we can prove that 
				\begin{align*}
					\mathcal{M}_{\sigma}f(z)\leq C(\sigma) \widetilde{\mathcal{M}}_{\sigma}f(z).
				\end{align*}
				
				Therefore,
				\begin{align*}
					C(\sigma)^{-1}\mathcal{M}_{\sigma}f(z) \leq \widetilde{\mathcal{M}}_{\sigma}f(z) \leq C(\sigma) \mathcal{M}_{\sigma}f(z). 	
				\end{align*} 
				Note that $C(\sigma)$ is a polynomial in $\sigma$ of degree atmost $2N$.
			\end{proof}

			Proof of Lemma \ref{lem:compar-maximal and atomic decomp}:
			\begin{proof}
				Using the partition of unity $\{\zeta_{j}\}$, we write $f$ as $f=\sum_{j}f\zeta_{j}$. Since $\mathcal{M}_{\sigma}$ is sublinear, we have $\mathcal{M}_{\sigma}f(z)\leq \sum_{j} \mathcal{M}_{\sigma}(f\zeta_{j})(z)$. This implies $\|\mathcal{M}_{\sigma}f\|
				_{p}^p\leq \sum_{j} \|\mathcal{M}_{\sigma}(f\zeta_{j})\|_{p}^p$. Therefore, by Lemma \ref{lem:compare-maximal-function}, we get $\|\mathcal{M}_{\sigma}f\|_{p}^p\leq C(\sigma) \sum_{j} \|g_{j}\|_{h^p_{\sigma}}^p$. 
				
				Now, we prove the converse part. For any $\varphi \in \mathcal{S}_{N}$, we can write
				\begin{align*}
					\left(f\zeta_{j}\right)\times\varphi_{t}(z)= {f\times(\psi_t^{z})_{t}(z)},
				\end{align*}
				where $\psi_t^{z}(w)= \varphi(w) \zeta_{j}(z-tw)$. Since $0<t<\sigma$ and $\supp \zeta_{j} \subset Q(z_{j}, \sigma)$, there exist a constant $C^{'}$ independent of $\sigma$ and $j$ such that $C^{'}\psi_t^{z}\in \mathcal{S}_{N}$ for all $0<t<\sigma$. This shows that $\mathcal{M}_{\sigma}(f\zeta_{j})(z) \leq C^{'} \mathcal{M}_{\sigma}f(z)$. Since $\mathcal{M}_{\sigma}f(z)\in L^{p}(\mathbb{C}^n)$, by Lemma \ref{lem:compare-maximal-function}, we get that $g_{j}\in  {h^p_{\sigma}(\mathbb{C}^n)}$ and 
				\begin{align*}
					\sum_{j} \|g_{j}\|_{h^p_{\sigma}}^p\leq & C(\sigma) \sum_{j} \|\mathcal{M}_{\sigma}(f\zeta_{j})\|_{p}^p\\
					\leq & C(\sigma)\sum_{j} \int \mathcal{M}_{\sigma}f(z)^p \chi_{E_{j}}(z)\, dz
				\end{align*}
				where $E_{j}$ is the set where $\mathcal{M}_{\sigma}(f\zeta_{j})$ is supported. Infact $E_j\subset Q(z_j,2\sigma)$. Since any point in $\mathbb{C}^n$ belongs to at most a finite number of such sets independent of $f$, we get
				\begin{align*}
					\sum_{j} \|g_{j}\|_{h^p_{\sigma}}^p\lesssim_{n}C(\sigma) \|\mathcal{M}_{\sigma}(f)\|_{p}^p.
				\end{align*}
				This completes the proof of the Lemma \ref{lem:compar-maximal and atomic decomp}.
			\end{proof}
			
			Proof of Lemma \ref{lem:comp-N-max-non-tang}.
			
			\begin{proof}
				We shall first show that for $w\in \mathbb{C}^n$, $t>0$, $N\geq0$,
				\begin{align} \label{eq:point-comp-N-max-non-tang}
					\left| e^{-t^2\mathcal{L}}f(z-w)\right|^{p} \left( 1+\frac{|w|}{t}\right)^{-Np} \leq \sum_{k=0}^{\infty}2^{(1-k)Np} \sup_{|w|<2^k t} \left| e^{-t^2\mathcal{L}}f(z-w)\right|. 
				\end{align}	
				To see this, when $|w|<t$, the term $k=0$ of the right side of \eqref{eq:point-comp-N-max-non-tang} dominates the left side. For $k\geq 1$, when $2^{k-1}t\leq |w| < 2^{k}t$, then $k$-th term of the right side dominates the left side of \eqref{eq:point-comp-N-max-non-tang}. This completes the proof of \eqref{eq:point-comp-N-max-non-tang}. 
				
				Now, from the estimate (25) of Chapter II of \cite{Stein-book-1993}, we get 
				\begin{align*}
					|\{ z: \sup_{|w|<2^k t} \left| e^{-t^2\mathcal{L}}f(z-w)\right| > \rho \}| \leq C (1+2^{k})^{2n} |\{ z: \sup_{|w|< t} \left| e^{-t^2\mathcal{L}}f(z-w)\right| > \rho \}|\,\, 
				\end{align*}	
				for all $\rho>0$ and $C=c_{2n}$. 

				The above estimate implies 
				\begin{align}\label{eq:dilated-nontan-comp-nontan}
					\int_{\mathbb{C}^n} \left( \sup_{|w|< 2^{k}t} \left| e^{-t^2\mathcal{L}}f(z-w)\right| \right)^{p} \, dz \leq C (1+2^{k})^{2n} \int_{\mathbb{C}^n} \left( \sup_{|w|< t} \left| e^{-t^2\mathcal{L}}f(z-w)\right| \right)^{p} \, dz .
				\end{align}
				
				Now, using \eqref{eq:point-comp-N-max-non-tang} and \eqref{eq:dilated-nontan-comp-nontan} and choosing $N>2n/p$, we get
				\begin{align*}
					\|M^{**}_{\CL,N} f\|_{p}  \leq C \sum_{k=0}^{\infty} &  (1+2^{k})^{2n} 2^{(1-k)Np} \|M^{*}_{\CL} f\|_{p}\\
					& \leq C_{N,p} \|M^{*}_{\CL} f\|_{p}.
				\end{align*}
				The last inequality follows from the fact that $N>2n/p$. This completes the proof of the Lemma \ref{lem:comp-N-max-non-tang}. 
			\end{proof}

			{	
				\section*{Acknowledgments}
				The authors would like to thank The Anh Bui for going through the manuscript and giving useful comments to improve the presentation of the paper.					 R.~Basak is supported by the National Postdoctoral Fellowship 
					(File no-PDF/2022/001443) from the Science and Engineering Research Board, Government of India and the National Science and Technology Council of Taiwan under research grant numbers 113-2811-M-003-007/113-2811-M-003-039.

			}	
			
			\bibliography{Hardy-wave-twisted-Laplacian-arxiv}
			\bibliographystyle{amsalpha}

		\end{document}